\documentclass[12pt,a4paper]{article}

\usepackage{natbib}  
\usepackage{color} 
\usepackage{hyperref} 
\usepackage{pstricks}
\usepackage{amssymb}
\usepackage{amsfonts,graphicx,amsmath,amsthm,color,nicefrac,enumerate, layout, bbm, listings}
\usepackage{vmargin}
\definecolor{lightblue}{rgb}{0.20,0.40,0.70}
\hypersetup{
breaklinks,
colorlinks=true,
linkcolor=lightblue,
urlcolor=lightblue,
citecolor=lightblue}

\newcommand{\fl}[1]{\lfloor #1 \rfloor_{h}}
\newcommand{\R}{\mathbb{R}}

\newcommand{\N}{\mathbb{N}}

\newcommand{\E}{\mathbb{E}}
\newcommand{\und}{\underline}
\newcommand{\Id}{\mathrm{Id}}

\newcommand{\co}{\colon}

\newcommand{\F}[0]{\ensuremath{\mathcal{F}}}
\newcommand{\one}[0]{\ensuremath{\mathbbm{1}}}

\newcommand{\smallsum}{\textstyle\sum}
\renewcommand{\P}{\mathbb{P}}

\newcommand{\lf}{\lfloor}
\newcommand{\rf}{\rfloor}

\newcommand{\Y}{\mathcal{X}}

\newcommand\numberthis{\addtocounter{equation}{1}\tag{\theequation}}

\theoremstyle{plain}
\newtheorem{theorem}{Theorem}[section]                                          
\newtheorem{prop}[theorem]{Proposition}                          
\newtheorem{lemma}[theorem]{Lemma}
\newtheorem{cor}[theorem]{Corollary}
\theoremstyle{definition}

\newtheorem{sett}[theorem]{Setting}
\theoremstyle{remark}

\let\inf\relax \DeclareMathOperator*\inf{\vphantom{p}inf}

\let\lim\relax \DeclareMathOperator*\lim{\vphantom{p}lim}

\makeatletter \@addtoreset{equation}{section} \makeatother
\renewcommand\theequation{\thesection.\arabic{equation}}

\begin{document}

\title{Existence, uniqueness, and numerical approximations
	 for stochastic Burgers equations}

\author{Sara Mazzonetto$^*$ and  Diyora Salimova$^{**}$
	\bigskip
	\\
	\small{$^*$Faculty of Mathematics, University of Duisburg-Essen,}\\
	\small{Germany, e-mail:   sara.mazzonetto@uni-due.de}
	\smallskip
	\\
	\small{$^{**}$Seminar for Applied Mathematics, ETH Zurich,}\\
	\small{Switzerland, e-mail:  diyora.salimova@sam.math.ethz.ch}}


\maketitle

\begin{abstract}
  In this paper we propose an all-in-one statement which includes  existence, uniqueness, regularity,  and numerical approximations of mild solutions 
  for a class of stochastic partial differential equations (SPDEs) with non-globally monotone nonlinearities.
	The proof of this result exploits the properties of an existent fully explicit space-time discrete approximation scheme and, in particular, the fact that it satisfies suitable a priori estimates. 
	As a byproduct we obtain almost sure and strong convergence of the approximation scheme to the mild solutions of the considered SPDEs.
	We conclude   by applying the main result of the paper to the stochastic Burgers equations with space-time white noise.
\end{abstract}


\section{Introduction}

In this work we exploit the properties of the approximation method introduced in~\citet{hutzenthaler2016strong} for a class of stochastic partial differential equations (SPDEs) with non-globally monotone nonlinearities driven by space-time white noise and
 obtain existence, uniqueness, and (spatial) regularity 
of the solution processes for such SPDEs.
At the same time, we achieve almost sure convergence of the approximation scheme (see Theorem~\ref{thrm:EU} below).
The proof  of the main result of the paper (see Theorem~\ref{thrm:EU} below) employs a priori estimates obtained in \citet[Corollary~2.6]{jentzen2019strong}
as well as an existence and uniqueness result for solutions of a class of Banach space valued evolution equations in \citet[Corollary~8.4]{jentzen2018existence}.
In addition, under the abstract setting of the main result, we  apply a strong convergence result in \citet[Theorem~3.5]{jentzen2019strong}, and thereby provide an all-in-one statement for existence, uniqueness, and (spatial) regularity of the solution processes and strong convergence of the approximation scheme in case of the considered SPDEs (see Corollary~\ref{cor:strong} below).

The approximation method we consider is the \emph{space-time full-discrete nonlinearity-truncated accelerated exponential Euler-type scheme} that converges strongly to the solutions of certain infinite-dimensional stochastic evolution equations with superlinearly growing non-linearities such as
stochastic Kuramoto-Sivashinsky equations with space-time white noise (see \citet[Corollary~5.10]{hutzenthaler2016strong}), 
stochastic Burgers equations and Allen-Cahn equations both driven by space-time white noise (see \citet[Corollary~5.6 \& Corollary~5.11]{jentzen2019strong}), and
two-dimensional stochastic Navier-Stokes equations driven by a certain  trace class noise (see \citet[Theorem~5.1]{mazzonetto2018strong}).
We would also like to mention that \citet[Theorem~1.1]{becker2017strong} establishes spatial and temporal rates of strong convergence for this scheme in the case of stochastic Allen-Cahn equations.

To explain our result better let us consider $H$ to be the real Hilbert space given by $H=L^{2}((0,1); \R)$,   $A \colon D(A) \subseteq H \to H$ to be the Laplace operator with Dirichlet boundary conditions on $H$, and $ ( H_r, \left< \cdot , \cdot \right>_{ H_r }, \left\| \cdot \right\|_{ H_r } ) $, $ r \in \R $,
to be a family of interpolation spaces associated to $ - A $.
The main result of this paper, Theorem~\ref{thrm:EU} below, is applicable to a subclass of stochastic evolution equations considered in Theorem~3.5 in \citet{jentzen2019strong}.
This subclass has to satisfy an additional regularity condition on the nonlinearity (see Setting~\ref{section:EUN:main}, in particular, inequality \eqref{eq:add:property}  below), 
which is crucial in the proof of pathwise a priori estimates for the approximation process (see Lemma~\ref{lem:apriori1} below).
These a priori bounds guarantee that the solution process takes values in an appropriate proper subspace of $H$, that is, $H_{\varrho}$ for some $\varrho \in (0, \infty)$, which determines the spatial regularity. We note that
Theorem~3.5 in  \cite{jentzen2019strong} requires that there exists a solution $X\colon \Omega \times [0,T] \to H_{\varrho}$ for some appropriate $\varrho \in [0,\infty)$.
Our main result establishes existence and uniqueness of the mild solution with a compatible spatial regularity. 
Techniques similar to the ones appearing in our proof can, e.g., be found in \citet[Theorem~3.1 and Subsection~4.3]{blomker2013galerkin}
which, in particular, provides existence and uniqueness of the mild solution for stochastic Burgers equations with space-time white noise with values in the Banach space $C((0,1),\R)$ 
exploiting spectral Galerkin approximations.

As an example, we choose to apply  the main result of this paper to the stochastic Burgers equations driven by space-time white noise. 
In this way for every $\varrho\in (\nicefrac18, \nicefrac14)$ we obtain the existence and uniqueness of the mild solution taking values in the space $H_{\varrho}$ (which is a  subspace of $C((0,1); \R)$).
We would  like to note that  there are several existence and uniqueness results in the literature for mild solutions of stochastic Burgers equations driven by colored noise (see, e.g., \citet{da1995stochastic}) and by space-time white noise (see, e.g., \citet{da1994stochastic} in the case of cylindrical Wiener process and \citet{bertini1994stochastic} in the case of Brownian sheet).
Other relevant references can, e.g., be found in \citet[Section~13.9]{da2014stochastic} and \citet[Chapter~14]{da1996ergodicity} and the references mentioned therein.
Our results extend the strong convergence result for stochastic Burgers equations in \cite[Corollary~5.6]{jentzen2019strong} because they yield existence, uniqueness, and spatial regularity of the mild solution and at the same time not only strong but also almost sure convergence for the numerical scheme. 

To conclude, let us mention the fact that our main all-in-one results (in particular, Corollary~\ref{cor:strong} below) can also be applied to the Kuramoto-Sivashinsky equations considered in \citet{hutzenthaler2016strong}, recovering the strong convergence result for the numerical scheme obtained there and also recovering the existence and uniqueness of the mild solution obtained in, e.g.,  \citet{duan2001stochastic}.

\subsection{Outline of the paper}
First, in Section~\ref{sec:pathwise}, we analyze pathwise regularity properties of the considered approximation scheme for a certain family of evolution equations. In particular, we obtain in Section~\ref{sec:pathwise} pathwise  a priori estimates and  convergence to a local mild solution (see Lemma~\ref{lem:apriori1} and Lemma~\ref{lem:convergence1}, respectively).
The non-explosion of the approximation scheme then leads to non-explosion of the unique maximal solution and therefore to pathwise existence and uniqueness of the global solution (see Proposition~\ref{prop:euT:Xomega}). 
The main result of the paper is given in Section~\ref{sec:main} in Theorem~\ref{thrm:EU}. It allows us to obtain an all-in-one statement for existence, uniqueness, and (spatial) regularity of the solution processes and strong convergence of the approximation scheme  in Corollary~\ref{cor:strong} below.
Finally, in Section~\ref{sec:burgers} we apply the latter to the stochastic Burgers equations with space-time white noise (see Corollary~\ref{cor:strong:burgers}).

\subsection{Notation}
Throughout this article the following notation is used. 
Let $ \mathbb{N} = \{1, 2, 3, \ldots \}$ be the set of all natural numbers. 
We denote by $ \lf \cdot \rf_h \colon \R \to \R$, $ h \in (0, \infty)$, the functions which satisfy for all $t \in \R$, $h \in (0, \infty)$ that 
\begin{align}
\lf t \rf_h = \max( (-\infty, t] \cap \{0, h, -h, 2h, -2h, \ldots\} ).
\end{align}
For a set $A$ we denote by $\#_{A}\in \N \cup \{0\} \cup \{\infty\}$ the number of elements of $A$ and we denote by $\Id_A \colon A \to A$ the function which satisfies for all $ a \in A$ that $\Id_A(a)=a$ 
(identity function on $ A $). 
For a topological space $ (X, \tau) $ 
we denote by $ \mathcal{B}(X) $ the Borel sigma-algebra of 
$ (X, \tau) $.

\section{Pathwise global solutions} \label{sec:pathwise}

This section is devoted to prove a pathwise existence of a unique global solution and convergence of the approximation scheme. We establish this result in Proposition~\ref{prop:euT:Xomega}. 
The main ingredients of the proof of Proposition~\ref{prop:euT:Xomega} are Lemma~\ref{lem:apriori1} and Lemma~\ref{lem:convergence1}. The latter establishes convergence and non-explosion of a (local) solution in a certain general setting and the former shows suitable a priori bounds for the (deterministic) approximation scheme. 

\begin{sett} \label{sett:apriori}
Let $ ( H, \left< \cdot , \cdot \right>_H, \left\| \cdot \right\|_H ) $ be a separable $\mathbb{R}$-Hilbert space, 
let $\mathbb{H}\subseteq{H}$ be a nonempty
orthonormal basis of $H$, 
let $\eta, \kappa \in [0,\infty)$,
let $\lambda \colon \mathbb{H} \to \mathbb{R}$ satisfy that $\inf_{b\in \mathbb{H}} \lambda_b>-\min\{\eta,\kappa\}$,
let
$ A \colon D(A) \subseteq H \to H $
be the linear operator which satisfies
$ D(A) = \{ v \in H \colon \sum_{b\in \mathbb{H}} | \lambda_b \langle b , v \rangle_H |^2 < \infty \} $
and
$ \forall \, v \in D(A) \colon A v = \sum_{b\in \mathbb{H}} - \lambda_b \langle b , v \rangle_H b$,
let $ ( H_r, \left< \cdot , \cdot \right>_{ H_r }, \left\| \cdot \right\|_{ H_r } ) $, 
$ r \in \R $, be a family of interpolation spaces associated to $ \kappa- A  $ (see, e.g., \citet[Section~3.7]{sell2002dynamics}), 
and let 
$ T,  \vartheta, c  \in (0,\infty)$,
$ \theta , \epsilon \in [0, \infty)$, 
$ \alpha, \varphi \in [0,1)$, 
$\gamma \in (0,1)$,
$\rho \in [-\alpha,1-\max\{\alpha,\gamma\})$, 
$ \varrho \in (\rho, 1-\gamma)$,
$\chi \in (0, \min\{(\varrho-\rho)/(1+\vartheta/2), (1-\alpha-\rho)/(1+\vartheta)\}]$.
\end{sett}

\subsection{A priori bounds}

\begin{lemma}[A priori bounds]\label{lem:apriori1}
Assume Setting~\ref{sett:apriori},
assume in addition that  $\sup_{b\in \mathbb{H}} |\lambda_b|$ $< \infty$,
let $\beta \in (0, \infty)$,
$ h \in (0, \min\{1,T\}]$,
and let $ Y, O, \mathbb{O} \colon [0, T] \to  H$, $F \in C(H, H)$, 
$\phi, \Phi \colon H \to [0,\infty)$  
satisfy for all $v, w \in H $, $t \in [0,T]$  that 
$
\eta O \in C([0,T],H)
$,
$ 
\mathbb{O}_t = O_t - \int_0^t e^{(t-s)(A-\eta)} \, \eta  O_s \, ds
$,
$
\|F(v)\|_{H_{-\gamma}} \leq c (2  \epsilon + \|v\|^2_H)
$,
$
\| F(v)\|_{H_{-\alpha}}^2 \leq \theta \max\{ 1, \| v\|_{ H_{\varrho} }^{2 + \vartheta} \}
$,
\begin{align}
\left< v, F( v + w ) \right>_{H} \leq \tfrac{1}{2} \phi(w) \| v \|^2_{H}+ \varphi \|(\eta-A)^{\nicefrac{1}{2}} v \|^2_{ H}+ \tfrac{1}{2}\Phi( w ),
\end{align} 
\begin{align} 
\begin{split}
& \| (\eta-A)^{-\nicefrac{1}{2}}(F(v) - F(w)) \|_H^2 \\
& \leq \theta \max\{ 1, \|v\|_{ H_{\varrho} }^{\vartheta} \} \|v-w\|_{ H_{\rho} }^2 + \theta \, \|v-w\|^{2 + \vartheta}_{ H_{\rho}},
\end{split}
\end{align} 
and
\begin{equation} \label{eq:apriori1:Y}
Y_t =  \int_0^t e^{ ( t - s ) A } \, \one_{[0, h^{ - \chi }]} \big( \big\|  Y_{ \fl{s} } \big\|_{H_{\varrho}} + \big\|  O_{ \fl{s} } \big\|_{H_{\varrho}}  \big)  F \big(  Y_{ \fl{s} } \big) \, ds + O_t.
\end{equation}
Then it holds that $\eta \mathbb{O} \in C([0,T],H)$ and for all $ t \in [0, T]$  that
\begin{align}
\begin{split}
 & 
\|Y_t-{O}_t\|_{H_\varrho} 
\leq  \frac{ 2c \, e^{t  \kappa} \, t^{(1-\varrho-\gamma)} }{(1- \varrho- \gamma)}    \Bigg( \epsilon + \sup_{s\in [0,T]} \| \mathbb{O}_{ \fl{s} }\|^2_H \, ds
	\\
	& 
 \quad + \left(1
+\tfrac{ \theta  e^{ \kappa \, (2+\vartheta)} [ 1+ (\kappa + \sqrt{\eta} +\sqrt{\eta}|\kappa-\eta|e^{ \eta} ) \|(\kappa-A)^{ \rho - \varrho } \|_{L(H)} + \sqrt{\theta} +  \sqrt{\eta} ]^{(2+\vartheta)}}
{(1- \varphi)( 1-\alpha - \rho)^{(2+\vartheta)}} \right)
	\\
	& 
 \qquad \cdot \int_0^T e^{ \int_s^T \phi( \mathbb{O}_{\fl{u} } )  +2\eta (1+\beta) \, du} \, \Big[  \Phi\big(\mathbb{O}_{ \fl{s} } \big)  + \tfrac{\eta}{2 \beta} \|\mathbb{O}_s\|_{H}^2  
	\\
	& 
\qquad + \left|\max \!\big\{  1 , \eta,T  \big\}\right|^{(4+3\vartheta)}  \max \!\Big\{ 1,  \smallint\nolimits_{0}^T   \|  \sqrt{ \eta} O_u \|^{(4+2\vartheta)}_{H_{\varrho}} \, du \Big\} \Big] \, ds \Bigg).
\end{split}
\end{align}
\end{lemma}
\begin{proof}[Proof of Lemma~\ref{lem:apriori1}] 
First, observe that for all $s\in(0,T)$ it holds that
\begin{align}
\begin{split}
 \| (\kappa-A)^{(\varrho + \gamma)} e^{ s  A } \|_{L(H)} 
 & = e^{s  \kappa} s^{-(\varrho+\gamma)} \| (s(\kappa-A))^{(\varrho + \gamma)} e^{ s (A-\kappa) } \|_{L(H)} \\
 &
 \leq e^{s \kappa} s^{-(\varrho+\gamma)}
\end{split}
\end{align}
(cf., e.g., Lemma~11.36 in \citet{renardy2006introduction}).
This,  \eqref{eq:apriori1:Y}, the triangle inequality, and the assumption that $\forall \, v \in H \colon \|F(v)\|_{H_{-\gamma}}\leq c (2 \epsilon +\|v\|_H^2)$ 
imply that for all $t\in [0,T]$ it holds that
\begin{align}
\begin{split}
 \|Y_t-{O}_t\|_{H_{\varrho}} 
	& 
\leq \int_0^t \| (\kappa-A)^{(\varrho + \gamma)} e^{ ( t - s ) A } \|_{L(H)} \|  F \big(  Y_{ \fl{s} } \big) \|_{H_{-\gamma}}\, ds 
	\\
	& 
\leq c \int_0^t  e^{( t - s ) \kappa} \left(t-s\right)^{-(\varrho+\gamma)} \left( 2 \epsilon + \|  Y_{ \fl{s} } \|^2_{H} \right) ds
	\\
	& 
\leq c \Big[ 2 \epsilon + \sup\nolimits_{s\in [0,t]}  \| Y_{ \fl{s} }\|^2_H \Big] \int_0^t e^{(t-s) \kappa} (t-s)^{-(\varrho+\gamma)} \, ds.
\end{split} 
\end{align}
This together with the fact that $\forall \, a, b \in \R \colon |a+b|^2 \leq 2 |a|^2 + 2|b|^2$ and the fact that $\varrho + \gamma <1$ shows that for all $t\in [0,T]$ it holds that
\begin{align*} \label{eq:apriori1:Hrhonorm}
	& \|Y_t-{O}_t\|_{H_{\varrho}} \\
	& \leq  c \!\left[ 2\epsilon 
	+ 2 \sup_{s\in [0,T]} \| Y_{ \fl{s} }-\mathbb{O}_{ \fl{s} } \|^2_H 
	+ 2 \sup_{s\in [0,T]} \| \mathbb{O}_{ \fl{s} }\|^2_H \right]  e^{t \kappa} \int_0^t  (t-s)^{-(\varrho+\gamma)} \, ds
	\\
	& 
=
2 c \! \left[ \epsilon 
+ \sup_{s\in [0,T]} \| Y_{ \fl{s} }-\mathbb{O}_{ \fl{s} } \|^2_H 
+ \sup_{s\in [0,T]} \| \mathbb{O}_{ \fl{s} }\|^2_H \right]  
\frac{e^{t \kappa} \, t^{(1-\varrho -\gamma)}}{(1-\varrho -\gamma)}. \numberthis
\end{align*}
Next note that  the assumption that $\sup_{b\in \mathbb{H}} |\lambda_b| < \infty$ 
assures that
$A\in L(H)$.
Corollary~2.6 in \citet{jentzen2019strong} therefore ensures that $\eta \mathbb{O} \in C([0,T],H)$ and   that
\begin{align*}
& \sup_{t\in [0,T]} \| Y_{\fl t} - \mathbb{O}_{\fl t} \|_{H}^2  
\leq    \int_0^T e^{ \int_s^T \phi( \mathbb{O}_{\fl{u} } )  +2\eta (1+\beta) \, du} \, \Big[  \Phi\big(\mathbb{O}_{ \fl{s} } \big)  + \tfrac{\eta}{2 \beta} \|\mathbb{O}_s\|_{H}^2    
	\\
	&
\quad + \tfrac{ \theta  e^{ \kappa(2+\vartheta)} [ 1+ (\kappa + \sqrt{\eta} +\sqrt{\eta}|\kappa-\eta|e^{ \eta} ) \|(\kappa-A)^{ \rho - \varrho } \|_{L(H)} + \sqrt{\theta} +  \sqrt{\eta} ]^{(2+\vartheta)} \left|\max \{1 , \smallint\nolimits_{0}^T   \|  \sqrt{\eta} O_u \|_{H_{\varrho}} \, du \}\right|^{(2+\vartheta)}}{(1- \varphi)( 1-\alpha - \rho)^{(2+\vartheta)}}   
	\\
	& 
\qquad \cdot  \max \!\big\{  h^{2(\varrho - \rho -\chi)},h^{ 2(1-\alpha - \rho -( 1 + \vartheta / 2 ) \chi)} , h \smallint\nolimits_{0}^T   \|  \sqrt{ \eta} O_u \|_{H_{\varrho}}^2 \, du \big\} 
	\\
& \qquad \cdot \left|\max \!\big\{ h^{-\chi},  \smallint\nolimits_{0}^T   \|  \sqrt{ \eta} O_u \|_{H_{\varrho}} \, du \big\}\right|^{\vartheta}\Big] \, ds \numberthis
	\\
	&  
\leq \left(1+\tfrac{ \theta  e^{ \kappa(2+\vartheta)}  [ 1+ (\kappa + \sqrt{\eta} +\sqrt{\eta}|\kappa-\eta|e^{ \eta} ) \|(\kappa-A)^{ \rho - \varrho } \|_{L(H)} + \sqrt{\theta} +  \sqrt{\eta} ]^{(2+\vartheta)}}{(1- \varphi)( 1-\alpha - \rho)^{(2+\vartheta)}} \right)
	\\
  & 
\quad	\cdot \int_0^T e^{ \int_s^T \phi( \mathbb{O}_{\fl{u} } )  +2\eta (1+\beta) \, du} \, \Big[  \Phi\big(\mathbb{O}_{ \fl{s} } \big)  + \tfrac{\eta}{2 \beta} \|\mathbb{O}_s\|_{H}^2  
	\\
	& \qquad  +  
\max \!\big\{  h^{2(\varrho - \rho -\chi)-\chi \vartheta},h^{ 2(1-\alpha - \rho -( 1 + \vartheta ) \chi)} , h^{1-\chi \vartheta} \smallint\nolimits_{0}^T   \|  \sqrt{ \eta} O_u \|_{H_{\varrho}}^2 \, du \big\} 
	\\
	& \qquad  \cdot 
\left|\max \{1 , \smallint\nolimits_{0}^T \|  \sqrt{\eta} O_u \|_{H_{\varrho}} \, du \}\right|^{(2+\vartheta)}   
\left|\max \!\big\{ 1, h^{\chi} \, \smallint\nolimits_{0}^T   \|  \sqrt{ \eta} O_u \|_{H_{\varrho}} \, du \big\}\right|^{\vartheta}\Big] \, ds.
\end{align*}
Combining this with the fact that $h<1$, $1-\chi \vartheta \geq 0$, $\tfrac{(1-\alpha-\rho)}{(1+\vartheta)}\geq \chi$, $\tfrac{(\varrho-\rho)}{(1+\nicefrac{\vartheta}{2})} \geq \chi$ demonstrates that for all $t\in [0,T]$ it holds that
\begin{align*} \label{eq:apriori1:Hnorm}
&
\sup_{t\in [0,T]} \| Y_{\fl t} - \mathbb{O}_{\fl t} \|_{H}^2 
	\\
	& 
\leq    \left(1+\tfrac{ \theta  e^{ \kappa(2+\vartheta)} [ 1+ (\kappa + \sqrt{\eta} +\sqrt{\eta}|\kappa-\eta|e^{ \eta} ) \|(\kappa-A)^{ \rho - \varrho } \|_{L(H)} + \sqrt{\theta} +  \sqrt{\eta} ]^{(2+\vartheta)}}{(1- \varphi)( 1-\alpha - \rho)^{(2+\vartheta)}} \right) \numberthis
	\\
	&
\quad \cdot \int_0^T e^{ \int_s^T \phi( \mathbb{O}_{\fl{u} } )  +2\eta (1+\beta) \, du} \, \Big[  \Phi\big(\mathbb{O}_{ \fl{s} } \big)  + \tfrac{\eta}{2 \beta} \|\mathbb{O}_s\|_{H}^2    
	\\
	& 
\qquad + \max \!\big\{  1 , \smallint\nolimits_{0}^T   \|  \sqrt{ \eta} O_u \|_{H_{\varrho}}^2 \, du \big\} \left|\max \!\big\{ 1,  \smallint\nolimits_{0}^T   \|  \sqrt{ \eta} O_u \|_{H_{\varrho}} \, du \big\}\right|^{(2+2\vartheta)}\Big] \, ds.
\end{align*} 
Moreover, note that H{\"o}lder's inequality implies that
\begin{align}
\begin{split}
& \max \!\big\{  1 , \smallint\nolimits_{0}^T   \|  \sqrt{ \eta} O_u \|_{H_{\varrho}}^2 \, d u \big\} \left|\max \!\big\{ 1,  \smallint\nolimits_{0}^T   \|  \sqrt{ \eta} O_u \|_{H_{\varrho}} \, du \big\}\right|^{2+2\vartheta} \\
	& \leq \max \!\big\{  1 , \eta \smallint\nolimits_{0}^T   \|  O_u \|_{H_{\varrho}}^2 \, d u \big\} \left|\max \!\big\{ 1,  \eta \, T \, \smallint\nolimits_{0}^T   \| O_u \|^2_{H_{\varrho}} \, du \big\}\right|^{1+\vartheta} \\
	& \leq 
\left|\max \!\big\{  1 , \eta, T  \big\}\right|^{3+2\vartheta} \left|\max \!\big\{ 1,  \smallint\nolimits_{0}^T   \| O_u \|^2_{H_{\varrho}} \, du \big\}\right|^{2+\vartheta}\\
	& \leq 
\left|\max \!\big\{  1 , \eta, T  \big\}\right|^{4+3\vartheta} \max \!\big\{ 1,  \smallint\nolimits_{0}^T   \| O_u \|^{2(2+\vartheta)}_{H_{\varrho}} \, du \big\}.
\end{split}
\end{align}
This together with \eqref{eq:apriori1:Hnorm} yields that
\begin{align} \label{eq:apriori1:Hnorm2}
\begin{split}
	& \sup_{t\in [0,T]} \| Y_{\fl t} - \mathbb{O}_{\fl t} \|_{H}^2   
	\\
	& 
\leq    \left(1+\tfrac{ \theta  e^{ \kappa(2+\vartheta)} [ 1+ (\kappa + \sqrt{\eta} +\sqrt{\eta}|\kappa-\eta|e^{ \eta} ) \|(\kappa-A)^{ \rho - \varrho } \|_{L(H)} + \sqrt{\theta} +  \sqrt{\eta} ]^{2+\vartheta}}{(1- \varphi)( 1-\alpha - \rho)^{2+\vartheta}} \right)
	\\
	&
\quad \cdot \int_0^T e^{ \int_s^T \phi( \mathbb{O}_{\fl{u} } )  +2\eta (1+\beta) \, du} \, \Big[  \Phi\big(\mathbb{O}_{ \fl{s} } \big)  + \tfrac{\eta}{2 \beta} \|\mathbb{O}_s\|_{H}^2    
	\\
	& \qquad + 
\left|\max \!\big\{  1 , \eta, T  \big\}\right|^{4+3\vartheta} \max \!\big\{ 1,  \smallint\nolimits_{0}^T   \| O_u \|^{2(2+\vartheta)}_{H_{\varrho}} \, du \big\} \Big] ds.
\end{split}
\end{align} 
Combining this and \eqref{eq:apriori1:Hrhonorm} 
completes the proof of Lemma~\ref{lem:apriori1}.
\end{proof}


\begin{lemma}[Pathwise convergence and non-explosion] \label{lem:convergence1}
Let $(V, \left\|\cdot\right\|_V)$ 
be a separable $\mathbb{R}$-Banach space,
let $(W,\left\|\cdot\right\|_W)$ 
be an $\mathbb{R}$-Banach space, 
let $T,  \chi \in (0,\infty)$, 
let $J\subseteq [0,T]$ be a convex set satisfying $0\in J$,
let $F\in C(V,W)$ and 
$\Psi \colon [0,\infty] \to [0,\infty]$ satisfy for all $r\in [0,\infty]$ that  $\Psi([0,\infty))\subseteq [0,\infty)$ and
\begin{align}
\begin{split}
 \Psi(r) 
= \sup \left( \left\{\tfrac{\|F(v)-F(w)\|_W}{\|v-w\|_V} \colon v, w \in V, v \neq w, \|v\|_V + \|w\|_V \leq r\right\} \cup \{0\}  \right),
\end{split}
\end{align}
let $S \colon (0,T) \to L(W,V)$ be a $\mathcal{B}((0,T))\slash \mathcal{B}(L(W,V))$-measurable function, 
let $\alpha\in [0,1)$ and $\left(P_n \right)_{n\in\N} \subseteq L(V)$ 
satisfy that
	$
	\sup_{s \in (0,T)} s^{\alpha} \|S_s\|_{L(W,V)} <\infty
	$,
	$
	\limsup_{ m \to \infty }$ $\| P_m \|_{ L(V) } < \infty
	$, 
	and 
	\begin{align}
	\limsup_{m\to\infty} \int_0^T \|( \Id_V - P_m ) S_s \|_{L(W,V)} \, d s =0,
	\end{align}
let $O\in C([0,T],V)$ and $\mathcal{O}^n \colon [0,T] \to V$, $n\in\N$, satisfy that 
\begin{align}
\limsup_{m\to\infty} \sup_{s \in [0, T]} \| O_s - \mathcal{O}_s^m\|_V=0,
\end{align}
let $\left(h_n\right)_{n\in\N} \subseteq (0,\infty)$ 
satisfy that $\limsup_{m\to\infty} h_m =0$,
and let $X \in C(J,V)$ and
$\Y^n \colon [0,T] \to V$, $n\in\N$, 
satisfy for all $t\in J$, $n\in \N$ that
	$
	X_t=\int_0^t S_{t-s}\, F(X_s) \, d s +O_t
	$, 
	\begin{align}
	\Y_{t}^n  = \int_0^{t} P_n \, S_{ t - s } \, \one_{[0,|h_n|^{-\chi}]} \big(  \big\| \Y_{ \lf s \rf_{h_n} }^n \big\|_{V} +  \big\|  \mathcal{O}_{ \lf s \rf_{h_n} }^n \big\|_{V}  \big) \, F \big(  \Y_{ \lf s \rf_{ h_n } }^n \big) \, ds  +\mathcal{O}_{t}^n,
	\end{align}
	and 
	$ 
	\liminf_{m\to \infty}\sup\nolimits_{s\in J} \|\Y_s^m\|_{V}<\infty.
	$ 
Then it holds 
\begin{enumerate}[(i)]
\item\label{item:lem:convergence1:1} for all $t\in J$ that $ \limsup_{n \to \infty} \sup_{s \in [0, t]} \| X_s- \Y_s^n \|_V=0 $ and
\item\label{item:lem:convergence1:2} that $\sup_{s\in J} \|X_s\|_V <\infty$.
\end{enumerate}
\end{lemma}
\begin{proof}[Proof of Lemma~\ref{lem:convergence1}]
First, observe that Proposition~3.3 in \citet{hutzenthaler2016strong} 
shows that for all $t \in J$  it holds that
\begin{equation}
\limsup_{n\to\infty} \sup_{s\in[0,t]} \|X_s-\Y^n_s\|_V=0.
\end{equation}
This establishes Item~\eqref{item:lem:convergence1:1}.
Next note that for all $n\in \N$, $t\in J$ it holds that
\begin{align}
\begin{split}
\|X_t\|_{V} 
	& \leq 
\|\mathcal{X}^n_t\|_{V} + \|X_t-\mathcal{X}^n_t\|_{V} 
	\\
	&\leq 
\sup\nolimits_{s\in J} \|\Y^n_s\|_{V} + \|X_t-\mathcal{X}^n_t\|_{V} 
	\\
	& \leq
\sup\nolimits_{s\in J} \|\Y^n_s\|_{V} + \sup\nolimits_{s\in [0,t]}\|X_s-\mathcal{X}^n_s\|_{V}
.
\end{split}
\end{align}
This together with Item~\eqref{item:lem:convergence1:1} implies that for all $t\in J$ it holds that
\begin{align}\label{eq:ndep}
\begin{split}
\|X_t\|_{V} 
	& \leq 
\liminf_{n\to \infty} \bigg(\sup_{s\in J} \|\Y^n_s\|_{V} + \sup_{s\in [0,t]}\|X_s-\mathcal{X}^n_s\|_{V} \bigg)
	\\
	& \leq
\liminf_{n\to \infty} \sup_{s\in J} \|\Y^n_s\|_{V} + \limsup_{n\to \infty} \sup_{s\in [0,t]}\|X_s-\mathcal{X}^n_s\|_{V} 
	\\
	& = 
\liminf_{n\to \infty} \sup_{s\in J} \|\mathcal{X}^n_s\|_{V} < \infty.
\end{split}
\end{align}
Therefore, we obtain that 
\begin{align}
\sup\nolimits_{t\in J} \|X_t\|_V \leq \liminf\nolimits_{n\to \infty} \sup\nolimits_{t\in J} \|\mathcal{X}^n_t\|_{V} < \infty.
\end{align} 
This establishes Item~\eqref{item:lem:convergence1:2}.
The proof of Lemma~\ref{lem:convergence1} is thus completed.
\end{proof}


\subsection{Pathwise existence, uniqueness, regularity, and approximation}


\begin{prop}[Global solutions]~\label{prop:euT:Xomega}
Assume Setting~\ref{sett:apriori},
	let $F \in C(H_{\varrho}, H_{-\alpha} )$,
	$\left(P_n\right)_{n\in\N} \subseteq L(H)$,
	let 
	$\mathbb{H}_n \subseteq \mathbb{H}$, $n\in\N$, be finite subsets of $\mathbb{H}$ satisfying for all $ n \in \N $, $u\in H$ that
	$ P_n(u) = \sum_{b\in \mathbb{H}_n} \langle b, u \rangle_H b $,
	let $ \phi, \Phi \colon H_1 \to [0,\infty)$ be functions
	such that
	for all $ n \in \N $, $ v, w \in P_n( H ) $ it holds
	that
	$F(v)\in H$,
	$
	\|F(v)\|_{H_{-\gamma}} \leq c \left(2  \epsilon + \|v\|^2_H\right),
	$
	\begin{align}
	\left< v, P_n F( v + w ) \right>_H \leq \phi( w ) \| v \|^2_H + \varphi \|  (\eta -A) ^{\nicefrac{1}{2}}v \|^2_{ H } + \Phi( w ),
	\end{align} 
	and
	\begin{align}  \label{eq:EUN:Flip}
	\left\| F(v) - F(w) \right\|_{ H_{ - \alpha } } \leq \theta \, ( 1 + \| v \|_{ H_{ \rho } }^{ \vartheta } + \|w\|_{H_{\rho}}^{\vartheta}) \, \|v-w\|_{H_{\rho}},
	\end{align}
let $\left(h_n \right)_{n\in\N} \subseteq (0, T]$ satisfy that $ \limsup_{ m \to \infty} h_m =0$,
 assume in addition that
 $ \alpha \in [0,\nicefrac{1}{2}]$,
 $ \varrho \in (\rho, 1-\max\{\alpha,\gamma\})$,
 $ \chi \in  (0, \min\{(\varrho-\rho)/(1+\vartheta), (1-\alpha-\rho)/(1+2\vartheta) \}] $, and
\begin{equation} \label{eq:limsup:hP:Pn}
\limsup_{ m \to \infty} \int_0^T \|( \Id_{H_{\varrho}} - P_m|_{H_{\varrho}} ) e^{s A} \|_{L(H_{-\alpha},H_{\varrho})} \, d s=0,
\end{equation} 
let $ O\in C([0,T],H_{\varrho})$ and
$ \mathcal{O}^n, \mathbb{O}^n \colon [0, T] \to H_{\varrho}$, $ n \in \N$, be functions which satisfy for all 
$ n \in \N $, $t\in [0,T]$
that
$\mathcal{O}^n([0,T])\subseteq P_n(H)$, 
$\eta \mathcal{O}^n \in C([0,T],P_n(H))$, 
\begin{equation} \label{eq:limsup:hP:O}
\limsup_{ m \to \infty} \sup_{s \in [0, T]} \| O_s - \mathcal{O}_s^m\|_{H_{\varrho}} =0,
\end{equation} 
$ \mathbb{O}^n_{t} = \mathcal{O}^n_{t} - \int_0^{t} e^{(t-s)(A-\eta)} \, \eta  \mathcal{O}^n_s \, ds$,
$\limsup_{m\to\infty} \sup_{ s \in [0,T]} \| \mathbb{O}_{ \lf s \rf_{h_m} }^m \|^2_{H} <\infty$, 
and
\begin{equation} \label{eq:limsup:fatO}
	\liminf_{m\to\infty} \int_0^T e^{ \int_r^T \, 2 \phi( \mathbb{O}_{\lf u \rf_{h_m} }^m )  \, du} \max\{ \Phi(\mathbb{O}_{ \lf r \rf_{h_m} }^m ), \|\mathbb{O}_r^m\|_{H}^2, 1, \smallint\nolimits_{0}^T \| \mathcal{O}_u^m \|_{H_{\varrho}}^{4+ 4\vartheta} \, du \} \, d r < \infty,
\end{equation}
and let
$ \Y^n \colon [0, T] \to H_{\varrho}$, $ n \in \N$, be functions satisfying
for all 
$ n \in \N $, $t \in [0,T]$ 
that
\begin{equation}
\Y_{t}^n = \int_0^{t} P_n \,  e^{  ( t - s ) A } \, \one_{ \{ \| \Y_{ \lf s \rf_{h_n} }^n \|_{ H_{\varrho} } + \| \mathcal{O}_{ \lf s \rf_{h_n} }^n\|_{ H_{\varrho} } \leq | h_n|^{ - \chi } \}} \, F \big(  \Y_{ \lf s \rf_{ h_n } }^n \big) \, ds + \mathcal{O}_{t}^n. 
\end{equation}
Then
\begin{enumerate}[(i)]
\item\label{item:bound:Xn} 
it holds that $\liminf_{n\to\infty}\sup_{s\in [0,T]}\|\Y^n_s\|_{H_{\varrho}}<\infty$,
\item \label{item:euT}
there exists a unique continuous function $X \colon [0,T] \to H_{\varrho}$ which satisfies for all $t \in [0,T]$ that $\int_0^t \|e^{ (t-s) A} \, F(X_s) \|_{H_{\varrho}} \, ds < \infty$ and
\begin{align}
 X_t = \int_0^t e^{ (t-s) A } \, F(X_s) \, ds + O_t,
\end{align}
and
\item \label{item:euT:pathwise} it holds that $\limsup_{n \to \infty} \sup_{t \in [0, T]} \| X_t- \Y_t^n \|_{H_{\varrho}}=0 $.

\end{enumerate}
\end{prop}
\begin{proof}[Proof of Proposition~\ref{prop:euT:Xomega}]
Observe that \eqref{eq:limsup:hP:Pn} allows us to assume w.l.o.g.~that for all $n \in \N$ it holds that $P_n(H) \neq \{0\}$.
Throughout this proof we assume that for all $n \in \N$ it holds that $P_n(H) \neq \{0\}$, let $\varepsilon \in (0,1-\alpha-\varrho)$ be a real number,  let $ \tilde{\theta} \in [0, \infty)$ be the real number given by
\begin{align*}
\tilde{\theta} 
& = 
\max\{1, \|(\eta-A)^{-1} (\kappa-A)\|_{L(H)}\}  \\
& \quad \cdot 
\max\!\bigg\{ \! \big(8 \theta^2 + 2 \, \| F(0) \|_{ H_{ - \alpha } }^2 \big) \max\!\bigg\{ 1, \sup_{
	u \in H_{ \varrho } \backslash \{ 0 \} }\tfrac{ \| u \|_{ H_{ \rho } }^{ 2 + 2 \vartheta } 
}{\| u \|_{ H_{ \varrho } }^{ 2 + 2 \vartheta } } \bigg\}, \numberthis
\\
& \qquad 3 \, \theta^2  \bigg[\sup_{ u \in H_{ - \alpha } \backslash\{ 0 \} }\tfrac{ \| u \|_{ H_{  \nicefrac{-1}{2} } }^2 }{\| u \|_{ H_{ - \alpha } }^2 } \bigg] \bigg[1 + \sup_{ u \in H_{ \varrho } \backslash \{ 0 \} }\tfrac{
	\| u \|_{ H_{ \rho } }^{ 2 \vartheta } }{\| u \|_{H_{\varrho}}^{2 \vartheta}} \bigg]
\big(1+2^{\max\{2\vartheta-1, 0\}}\big)  \! \bigg\},
\end{align*}
let $\Psi \colon [0,\infty] \to [0,\infty]$ be the function which satisfies for all $r\in [0,\infty]$  that 
$\Psi(r)=\sup (\{\|F(v)-F(w)\|_{H_{-\alpha}}/\|v-w\|_{H_{\varrho}} \colon v, w \in H_{\varrho}, v\neq w, \|v\|_{H_{\varrho}}+\|w\|_{H_{\varrho}} \leq r \} \cup \{0\}) $, and let
$\psi \colon (0,T)\to (0,\infty)$
be the function which satisfies for all $t\in (0,T)$ that 
\begin{align} \label{prop:euT:phi}
\psi(t)=e^{t \, \kappa}-1 + t^\varepsilon.
\end{align}
Note that \eqref{eq:EUN:Flip} ensures that for all $r \in [0,\infty)$, $v, w \in H_{\varrho}$ satisfying $v\neq w$ and $\|v\|_{H_{\varrho}} + \|w\|_{\varrho} \leq r$ it holds that
\begin{align}
\begin{split} 
\tfrac{\|F(v)-F(w)\|_{H_{-\alpha}}}{\|v-w\|_{H_{\varrho}}} 
	& 
	\leq 
\tfrac{\|v-w\|_{H_{\rho}}}{\|v-w\|_{H_{\varrho}}} \, \theta \, ( 1 + \| v \|_{ H_{ \rho } }^{ \vartheta } + \|w\|_{H_{\rho}}^{\vartheta}) 
	 \\
	 & \leq   \left[ \sup\nolimits_{u\in H_{\varrho}\setminus\{0\}} \tfrac{\|u\|_{H_{\rho}}}{\|u\|_{H_{\varrho}}} \right]  \theta \, (1+2 r^{\vartheta})  <\infty.
\end{split}
\end{align}
Therefore, we obtain that for all $r\in [0,\infty)$ it holds that 
\begin{equation}
\label{eq:bounded_Psi}
\Psi(r)\leq  \theta \, (1+ 2 r^{\vartheta}) \left[ \sup\nolimits_{u\in H_{\varrho}\setminus\{0\}} \tfrac{\|u\|_{H_{\rho}}}{\|u\|_{H_{\varrho}}} \right] <\infty.
\end{equation} 
This establishes that
\begin{align}
\Psi([0,\infty))\subseteq [0,\infty).
\end{align}
Next observe that for all $n\in \N$, $t\in [0,T]$ it holds that
\begin{align}
\|\Y^n_t\|_{H_\varrho} \leq \|O_t\|_{H_{\varrho}} + \|\mathcal{O}^n_t - O_t\|_{H_{\varrho}} + \|\Y^n_t-\mathcal{O}^n_t\|_{H_{\varrho}}.
\end{align}
This and \eqref{eq:limsup:hP:O} yield that
\begin{align} \label{eq:disXn}
\begin{split}
\liminf_{n\to \infty} \sup_{t\in [0,T]} & \|\Y^n_t\|_{H_\varrho} \leq 
 \sup_{t\in [0,T]}\|O_t\|_{H_{\varrho}} 
 + \liminf_{n\to\infty}\sup_{t\in [0,T]}\|\Y^n_t-\mathcal{O}^n_t\|_{H_{\varrho}}.
\end{split}
\end{align}
Furthermore, note that \eqref{eq:EUN:Flip} and, e.g., Lemma~2.4 in \citet{hutzenthaler2016strong} (with 
$V=H_{\varrho}$, 
$\mathcal{V}=H_{\rho}$, 
$W=H_{-\alpha}$, 
$\mathcal{W}= H_{-\nicefrac{1}{2}}$, 
$\epsilon=\theta$, 
$\theta=  \max\{1,\|(\eta -A)^{-1} (\kappa -A)\|_{L(H)}\}^{-1} \, \tilde{\theta}$, 
$\varepsilon= \vartheta$, 
$\vartheta = 2 \vartheta$, 
in the notation of Lemma~2.4 in \citet{hutzenthaler2016strong}) ensures for all $v, w \in H_{\varrho}$  that 
\begin{align}
\label{eq:Fcondition:apriori}
\begin{split}
& \|(\eta-A)^{-\nicefrac{1}{2}}(F(v)-F(w))\|^2_{H} \\
& \leq  \| (\eta -A)^{-1} (\kappa -A) \|_{L(H)} \|F(v)-F(w)\|^2_{H_{-\nicefrac{1}{2}}}\\
 & \leq  \tilde{\theta} \left(\max\{1, \|v\|_{H_{\varrho}}^{{2\vartheta}}\}\|v-w\|_{H_{\rho}}^2 +  \|v-w\|_{H_{\rho}}^{2+{2\vartheta}}\right)
\end{split}
\end{align}
and  
\begin{align}
\label{eq:Fcondition2}
\|F(v)\|_{H_{-\alpha}}^2 \leq \tilde\theta \max\{ 1, \|v\|_{H_{\varrho}}^{2 + 2 \vartheta} \}.
\end{align}
In addition,  observe that the assumption that  $\forall \, n \in \N \colon \mathcal{O}^n([0,T]) \subseteq P_n(H)$ implies for all $n\in \N$ that $\mathbb{O}^n([0,T]) \cup \mathcal{X}^n([0,T]) \subseteq P_n(H)$.
Combining this, \eqref{eq:Fcondition:apriori},  \eqref{eq:Fcondition2}, and Lemma~\ref{lem:apriori1}
 (with
 $ H=P_n(H)$,
 $ \beta=1$,
 $ \theta= \tilde{\theta}$, 
 $ \vartheta = {2\vartheta}$, 
 $ A=(P_n(H) \ni v\mapsto Av \in P_n(H) )\in L(P_n(H))$,
 $ h=h_n$,
 $ Y = ( [0,T] \ni t \mapsto {\Y}^{n}_t \in P_n(H) )$, 
 $ \mathbb{O}= ( [0,T] \ni t \mapsto \mathbb{O}^{n}_t( \omega ) \in P_n(H))$, 
 $ \mathcal{O} = ( [0,T] \ni t \mapsto \mathcal{O}^{n}_t( \omega ) \in P_n(H) )$,
 $ F = ( P_n(H) \ni v\mapsto P_n F(v) \in  P_n(H) \cap H_{-\alpha} )\in C(P_n(H), P_n(H))$,
 $ \phi = 2 \phi |_{P_n(H)}$, $ \Phi = 2 \Phi |_{P_n(H)} $ 
 for $n\in \{ m \in \N \colon h_m \leq 1 \}$
 in the notation of Lemma~\ref{lem:apriori1})
 yields that for all $n\in  \{ m \in \N \colon h_m \leq 1 \}$  it holds that
\begin{align}
\begin{split}
& \sup\nolimits_{t\in [0,T]} \|\Y^n_t-\mathcal{O}^n_t\|_{H_\varrho}
	 \leq 
\frac{ 2c \, e^{T  \kappa} \, T^{(1-\varrho-\gamma)} }{(1- \varrho- \gamma)}    \Bigg(  \epsilon+ \sup\nolimits_{s\in [0,T]} \| \mathbb{O}^n_{ \lfloor s \rfloor_{h_n} }\|^2_H 
	\\
& \left. + \left(1+\tfrac{ \tilde\theta  e^{ \kappa(2+2\vartheta)} [ 1+ (\kappa + \sqrt{\eta} +\sqrt{\eta}|\kappa-\eta|e^{ \eta} ) \|(\kappa-A)^{ \rho - \varrho } \|_{L(H)} + \sqrt{\tilde\theta} +  \sqrt{\eta} ]^{2+2\vartheta}}{(1- \varphi)( 1-\alpha - \rho)^{2+2\vartheta}} \right) \right.
	\\
& \left. \quad \cdot \int_0^T e^{ \int_s^T 2 \phi( \mathbb{O}^n_{\lfloor u \rfloor_{h_n} } )  +4\eta \, du} \, \Big[  2\Phi\big(\mathbb{O}^n_{ \lfloor s \rfloor_{h_n}  } \big)  + \tfrac{\eta}{2} \|\mathbb{O}^n_s\|_{H}^2 \right.   
	\\
&  \qquad + \left|\max \!\big\{  1 , \eta,T  \big\}\right|^{4+6\vartheta} 
\max \!\big\{ 1,  \smallint\nolimits_{0}^T   \|  \sqrt{ \eta} \mathcal{O}^n_u \|^{4+4\vartheta}_{H_{\varrho}} \, du \big\} \Big] \, ds \Bigg).
\end{split}
\end{align}
Hence, we obtain that
\begin{align*}\label{eq:diff:YO}
& \liminf_{n\to\infty} \sup\nolimits_{t\in [0,T]} \|\Y^n_t-\mathcal{O}^n_t\|_{H_\varrho}
 \leq 
\frac{ 2c \, e^{T  \kappa} \, T^{(1-\varrho-\gamma)} }{(1- \varrho- \gamma)}    \Bigg(  \epsilon+ \limsup_{n\to\infty}\sup\nolimits_{s\in [0,T]} \| \mathbb{O}^n_{ \lfloor s \rfloor_{h_n} }\|^2_H 
	\\
& \quad \left. + \left(1+\tfrac{ \tilde\theta  e^{ \kappa(2+2\vartheta)} [ 1+ (\kappa + \sqrt{\eta} +\sqrt{\eta}|\kappa-\eta|e^{ \eta} ) \|(\kappa-A)^{ \rho - \varrho } \|_{L(H)} + \sqrt{\tilde\theta} +  \sqrt{\eta} ]^{2+2\vartheta}}{(1- \varphi)( 1-\alpha - \rho)^{2+2\vartheta}} \right) \right.
	\\
& \left. \qquad \cdot  \, \liminf_{n\to\infty} \int_0^T e^{ \int_s^T 2 \phi( \mathbb{O}^n_{\lfloor u \rfloor_{h_n} } )  +4\eta \, du} \, \Big[  2\Phi\big(\mathbb{O}^n_{ \lfloor s \rfloor_{h_n}  } \big)  + \tfrac{\eta}{2} \|\mathbb{O}^n_s\|_{H}^2 \right.   
	\\
&  \qquad \quad + \left|\max \!\big\{  1 , \eta,T  \big\}\right|^{4+6\vartheta} 
\max \!\big\{ 1,  \smallint\nolimits_{0}^T   \|  \sqrt{ \eta} \mathcal{O}^n_u \|^{4+4\vartheta}_{H_{\varrho}} \, du \big\} \Big] \, ds \Bigg). \numberthis
\end{align*}
Combining this,
 the assumption that $\limsup_{m\to \infty} \sup_{s\in [0,T]} \| \mathbb{O}^m_{ \lf s \rf_{h_m} } \|^2_H<\infty$, and \eqref{eq:limsup:fatO} 
assures that 
\begin{align}
\liminf_{n\to\infty} \sup_{t\in [0,T]} \|\Y_t^n-\mathcal{O}^n_t\|_{H_\varrho}<\infty.
\end{align}
The assumption that $O\in C([0,T], H_{\varrho})$ and \eqref{eq:disXn} therefore prove that 
\begin{align}
\liminf_{n\to\infty} \sup_{t\in [0,T]} \| \Y^n_t \|_{H_\varrho}<\infty.
\end{align}
This establishes Item~\eqref{item:bound:Xn}.
%
%
In the next step we observe that \eqref{prop:euT:phi} yields that
\begin{align}
\label{eq:limsup:0}
\limsup_{t \searrow 0} \psi(t)=0.
\end{align}
Moreover, note that the fact that  $ \forall \, r\in [0,1], t \in (0,T) \colon \|(t (\kappa-A))^r \, e^{t A}\|_{L(H)}\leq e^{t  \kappa}$ 
and
$\| (t (\kappa-A))^{-r} \left(e^{ t (A-\kappa) }- \Id_{H}\right)\|_{L(H)} \leq 1$ (cf., e.g., Lemma~11.36 in \citet{renardy2006introduction}) implies that for all $s\in [0,T), t\in (s,T]$ it holds that
\begin{align*} 
\label{eq:sup_dist_etA} 
& s^{(\alpha+\varrho+\varepsilon)} \, \|e^{ t A } - e^{ s A } \|_{L(H_{-\alpha},H_{\varrho})}  = s^{\varepsilon} \left\|(s(\kappa-A))^{(\alpha+\varrho)} e^{s A}\left(e^{ (t-s) A } - \Id_{H} \right) \right\|_{L(H)}\\
& \leq s^{\varepsilon} \left\|(s(\kappa-A))^{(\alpha+\varrho)} e^{s A}\left(e^{ (t-s) A } -  e^{ (t-s) \kappa }  \right) \right\|_{L(H)} \\
& \quad + s^{\varepsilon} \left\|(s(\kappa-A))^{(\alpha+\varrho)} e^{s A}\left(  e^{ (t-s) \kappa } - \Id_{H} \right) \right\|_{L(H)} \\
& \leq e^{ (t-s) \kappa } \left\|(s(\kappa-A))^{(\alpha+\varrho+\varepsilon)} e^{s A} \right\|_{L(H)} \left\|\left(\kappa-A\right)^{-\varepsilon} \left(e^{ (t-s) (A-\kappa) } - \Id_{H} \right)  \right\|_{L(H)}\\
& \quad + s^{\varepsilon} \left(e^{ (t-s) \kappa } - 1 \right) \left\|(s(\kappa-A))^{(\alpha+\varrho)} e^{s A} \right\|_{L(H)}\\
& \leq e^{t \kappa} \, (t-s)^\varepsilon + s^{\varepsilon}  \left(e^{ (t-s) \kappa } - 1 \right)   e^{s  \kappa}  \\
& \leq  \max\{1, T^\varepsilon\}\, e^{T  \kappa} \left(e^{ (t-s) \kappa } - 1 + (t-s)^\varepsilon \right) \! . \numberthis
\end{align*} 
and
\begin{align}
\label{eq:sup_norm_etA}
s^{\alpha+\varrho +\varepsilon} \, \|e^{sA} \|_{L(H_{-\alpha},H_{\varrho})}  =  s^{\varepsilon} \left\|(s(\kappa-A))^{(\alpha+\varrho)}e^{sA} \right\|_{L(H)} \leq s^{\varepsilon} e^{s \kappa} \leq T^{\varepsilon} e^{T \kappa} .
\end{align}
This together with \eqref{prop:euT:phi} yields that
\begin{align} 
\begin{split}
& \sup_{s \in (0,T)}  \left[ s^{\alpha+\varrho +\varepsilon} \left( \|e^{sA} \|_{L(H_{-\alpha},H_{\varrho})}  + \sup_{ t \in (s,T)} \frac{\|e^{ t A } - e^{ s A } \|_{L(H_{-\alpha},H_{\varrho})} }{|\psi(t-s)|} \right) \right] 
\\ 
&\leq 2 e^{T  \kappa} \max\{1,T^\varepsilon\} < \infty.
\end{split}
\end{align}
Combining this, \eqref{eq:bounded_Psi}, \eqref{eq:limsup:0}, and Item~(i) in Corollary~8.4 in \citet{jentzen2018existence} 
(with 
$(V,  \left\| \cdot \right\|_V) = (H_{\varrho}, \left\| \cdot \right\|_{H_{\varrho}})$, 
$(W,  \left\| \cdot \right\|_W) = (H_{-\alpha}, \left\| \cdot \right\|_{H_{-\alpha}})$, 
$S=\big((0, T) \ni t \mapsto ( H_{ - \alpha } \ni v \mapsto e^{ t A } v \in H_{ \varrho } ) \in L(H_{-\alpha},H_{\varrho} ) \big)$,
$\mathcal{S}=\big([0, T] \ni t \mapsto ( H_{ \varrho } \ni v \mapsto e^{ t A } v \in H_{ \varrho } ) \in L(H_{\varrho}) \big)$,
$o=O$,
$\phi=\psi$
in the notation of Corollary~8.4 in \citet{jentzen2018existence}) 
demonstrates that there exists a convex set $J \subseteq [0,T]$ with $\{0\} \subsetneq J$ such that there exists a unique continuous function $ X \colon J \to H_{\varrho}$ 
which satisfies for all $t\in J$ that
\begin{align}
\label{eq:bound:BochnerX}
\int_0^t \|e^{ (t-s) A} \, F(X_s) \|_{H_{\varrho}} \, ds < \infty , \qquad X_t = \int_0^t e^{ (t-s) A } \, F(X_s) \, ds + O_t,
\end{align}
and
\begin{align}
\label{eq:bound:TX}
\limsup\nolimits_{s \nearrow \sup(J)} \left[\tfrac{1}{(T-s)}+ \|X_s\|_{H_\varrho}\right]= \infty.
\end{align}
Next observe that Item~\eqref{item:bound:Xn} ensures that 
$
\liminf_{n\to\infty}\sup_{s\in J}\|\Y^n_s\|_{H_{\varrho}}<\infty
$.
Lemma~\ref{lem:convergence1} (with 
$(V,  \left\| \cdot \right\|_V) = (H_{\varrho}, \left\| \cdot \right\|_{H_{\varrho}})$, 
$(W,  \left\| \cdot \right\|_W) = (H_{-\alpha}, \left\| \cdot \right\|_{H_{-\alpha}})$, 
$ \alpha = \varrho + \alpha $, 
$ S = \big((0,T] \ni t \mapsto ( H_{ - \alpha } \ni v \mapsto e^{ t A } v \in H_{ \varrho } ) \in L(H_{-\alpha},H_{\varrho})\big) $, 
$ (P_n)_{n \in \N} = ( H_{ \varrho } \ni v \mapsto P_n( v ) \in H_{ \varrho } )_{n \in \N} $
in the notation of Lemma~\ref{lem:convergence1}) 
hence shows that for all $t\in J$ it holds that
\begin{align} \label{eq:bound:term2}
\limsup_{n\to \infty} \sup_{s\in [0,t]}\|X_s-\Y^n_s\|_{H_{\varrho}} =0.
\end{align}
This, in particular,  
implies that $\sup_{s\in J}\|X_s\|_{H_\varrho}<\infty$.
Item~(iii) in Corollary~9.4 in \citet{jentzen2018existence}
therefore assures that $J=[0,T]$. 
This together with \eqref{eq:bound:BochnerX} establishes Item~\eqref{item:euT}. Next observe that the fact that $T\in J$ and \eqref{eq:bound:term2} prove Item~\eqref{item:euT:pathwise}.
The proof of Proposition~\ref{prop:euT:Xomega} is thus completed.
\end{proof}

\section{The main result: Existence, uniqueness, and strong convergence} \label{sec:main}
In this section we accomplish in Theorem~\ref{thrm:EU}
global existence and  uniqueness of the solutions for certain class of SPDEs.  Moreover, Theorem~\ref{thrm:EU} shows an almost sure convergence 
of the approximation scheme~\eqref{boldface_X} below.
The other result of this section is Corollary~\ref{cor:strong}, which establishes a strong convergence of the approximation scheme and follows from Theorem~\ref{thrm:EU} and \citet[Theorem~3.5]{jentzen2019strong}.

\begin{sett} \label{section:EUN:main}
	Let $ ( H, \left< \cdot , \cdot \right>_H, \left\| \cdot \right\|_H ) $ be a separable $\mathbb{R}$-Hilbert space,
	let $\mathbb{H}\subseteq{H}$ be a nonempty
	orthonormal basis of $H$, 
	let $\eta, \kappa \in [0,\infty)$,
	let $\lambda \colon \mathbb{H} \to \mathbb{R}$ satisfy that $\inf_{b\in \mathbb{H}} \lambda_b>-\min\{\eta,\kappa\}$,
	let
	$ A \colon D(A) \subseteq H \to H $
	be the linear operator which satisfies
	$ D(A) = \{ v \in H \colon \sum_{b\in \mathbb{H}} | \lambda_b \langle b , v \rangle_H |^2 < \infty \} $
	and
	$ \forall \, v \in D(A) \colon A v = \sum_{b\in \mathbb{H}} - \lambda_b \langle b , v \rangle_H b$,
	let $ ( H_r, \left< \cdot , \cdot \right>_{ H_r }, \left\| \cdot \right\|_{ H_r } ) $, 
	$ r \in \R $, be a family of interpolation spaces associated to $ \kappa- A  $ (see, e.g., \citet[Section~3.7]{sell2002dynamics}), 
	 let 
	$ T,  \vartheta, c  \in (0,\infty)$,
	$ \theta , \epsilon \in [0, \infty)$, 
	$ \alpha \in [0,\nicefrac{1}{2}]$,
	$  \varphi \in [0,1)$, 
	$\gamma \in (0,1)$,
	$\rho \in [-\alpha,1-\max\{\alpha,\gamma\})$, 
$ \varrho \in (\rho, 1-\max\{\alpha,\gamma\})$,
$ \chi \in  (0, \min\{(\varrho-\rho)/(1+\vartheta), (1-\alpha-\rho)/(1+2\vartheta) \}] $,
	let $F \in C(H_{\varrho}, H_{-\alpha} )$,
	$\left(P_n\right)_{n\in\N} \subseteq L(H)$,
	let 
	$\mathbb{H}_n \subseteq \mathbb{H}$, $n\in\N$, be finite subsets of $\mathbb{H}$ satisfying for all $ n \in \N $, $u\in H$ that
	$ P_n(u) = \sum_{b\in \mathbb{H}_n} \langle b, u \rangle_H b $ and $ \liminf_{ m \to \infty } \inf( \{\lambda_b \colon b \in \mathbb{H} \backslash \mathbb{H}_m \} \cup \{\infty\}  ) = \infty $,
	let $ \phi, \Phi \colon H_1 \to [0,\infty)$ be functions
	such that
	for all $ n \in \N $, $ v, w \in P_n( H ) $ it holds
	that
	$F(v)\in H$,
	\begin{align} \label{eq:add:property}
	\|F(v)\|_{H_{-\gamma}} \leq c \left(2  \epsilon + \|v\|^2_H\right),
	\end{align}
	\begin{align}
	\left< v, P_n F( v + w ) \right>_H \leq \phi( w ) \| v \|^2_H + \varphi \|  (\eta -A) ^{\nicefrac{1}{2}}v \|^2_{ H } + \Phi( w ),
	\end{align} 
	and
	\begin{align}  \label{eq:EUN:Flip:3}
	\left\| F(v) - F(w) \right\|_{ H_{ - \alpha } } \leq \theta \, ( 1 + \| v \|_{ H_{ \rho } }^{ \vartheta } + \|w\|_{H_{\rho}}^{\vartheta}) \, \|v-w\|_{H_{\rho}},
	\end{align}
	let $\left(h_n \right)_{n\in\N} \subseteq (0, T]$ satisfy that $ \limsup_{ m \to \infty} h_m =0$,
let $ ( \Omega, \F, \P ) $ be a probability space,
let $ \Y^n \colon [0, T] \times \Omega \to H_\varrho$, $ n \in \N$, be stochastic processes,
let $\mathcal{O}^n \colon [0, T] \times \Omega \to H_\varrho$, $n\in \N$, and $O  \colon [0, T] \times \Omega \to H_\varrho$ be stochastic processes with continuous sample paths, 
let $\mathbb{X}^n, \mathbb{O}^n \colon [0,T] \times \Omega \to  H_{\varrho}$, $n\in \N$, be functions,
and assume
for all $ n \in \N $, $ t \in [ 0, T ] $ 
that
\begin{align}
\label{boldface_X}
\mathbb{X}_t^n  =  \smallint\nolimits_0^t P_n \,  e^{  ( t - s ) A } \, \one_{ \{ \| \Y_{ \lf s \rf_{h_n} }^n \|_{ H_{ \varrho } } +  \| \mathcal{O}_{ \lf s \rf_{h_n} }^n \|_{ H_{ \varrho } }  \leq | h_n|^{ - \chi } \}} \, F \big(  \Y_{ \lf s \rf_{ h_n } }^n \big) \, ds  +\mathcal{O}_t^n,
\end{align} 
$ \mathcal{O}^n( [0,T] \times \Omega ) \subseteq P_n( H ) $, $\mathbb{O}_t^n = \mathcal{O}_t^n - \int_0^t e^{(t-s)(A-\eta)} \, \eta \mathcal{O}_s^n \, ds$, and
$\P (\mathbb{X}_t^n  = \Y_t^n )=1$.
\end{sett}

\begin{theorem}[Existence, uniqueness, and almost sure convergence]
\label{thrm:EU}
Assume Setting~\ref{section:EUN:main}, let $\Omega_0 \in \left\{ B\in \mathcal{F}\colon  \mathbb{P}(B)=1\right\}$,   and assume that for $\omega \in \Omega_0$ it holds that
\begin{multline} \label{eq:thrm:EU:limsup:O_int}
 \liminf_{ m \to \infty}  \int_0^T e^{ \int_r^T 2\, \phi( \mathbb{O}_{\lf u \rf_{h_m} }^m(\omega) )  \, du} \Big[ 1+ |\Phi(\mathbb{O}_{ \lf r \rf_{h_m} }^m (\omega))| + \smallint\nolimits_{0}^T \| \mathcal{O}_u^m (\omega)\|_{H_{\varrho}}^{4+ 4\vartheta} \, du \\
+ \|\mathbb{O}_r^m (\omega)\|_H^{2}  \Big] \, dr< \infty  
\end{multline}
and 
\begin{align}
\label{eq:O:convergence}
\limsup_{m  \to \infty} \sup_{t \in [0,T]}\| O_t(\omega)- \mathcal{O}_t^{m} (\omega) \|_{ H_{ \varrho } } = 0.
\end{align}
Then 
\begin{enumerate}[(i)]
	\item \label{item:solution} there exists an up-to-indistinguishability unique stochastic process $X  \colon$ $ [0, T] \times \Omega \to H_\varrho$ with continuous sample paths which satisfies  that for all $t\in [0,T]$ it holds $\P$-a.s.~that
	\begin{align} \label{item:solution:mild}
	X_t = \int_0^t e^{ ( t - s ) A } \, F( X_s ) \, ds + O_t
	\end{align}
 and
	\item \label{item:convergence:prob}  
	there exists an event 
	$
	\Omega_1\in \left\{ B\in \mathcal{F}\colon  \mathbb{P}(B)=1\right\}
	$ 
	such that for all $\omega\in \Omega_1$ it holds that
\begin{align}
\limsup_{n \to \infty} \sup_{t \in [0, T]} \| X_t(\omega)- \mathbb{X}_t^{n}(\omega) \|_{H_\varrho}=0.
\end{align} 
\end{enumerate}
\end{theorem}
\begin{proof}[Proof of Theorem~\ref{thrm:EU}]
Throughout this proof let $\Omega_1 \subseteq \Omega$ be the set given by
\begin{align}
\label{eq:Omega_1}
\Omega_1   = \Omega_0 \cap \big\{\omega \in \Omega \colon  (\forall \, m \in \N, s \in  [0, T] \colon \mathbb{X}_{\lf s \rf_{h_m}}^m(\omega)= \Y_{\lf s \rf_{h_m}}^m(\omega) )\big\}
\end{align}  
and
let ${\bf X}^n \colon [0,T] \times  \Omega \to H, n\in \N,$ be the functions which satisfy for all $n \in \N$, $t \in [0, T]$  that
\begin{align}
\label{eq:tilde_X}
{\bf X}_t^n  =  \int_0^t P_n \,  e^{  ( t - s ) A } \, \one_{ \{ \|{\bf X}_{ \lf s \rf_{h_n} }^n \|_{ H_{ \varrho } } +  \| \mathcal{O}_{ \lf s \rf_{h_n} }^n \|_{ H_{ \varrho } }  \leq | h_n|^{ - \chi } \}} \, F \big(  {\bf X}_{ \lf s \rf_{ h_n } }^n \big) \, ds  +\mathcal{O}_t^n.
\end{align} 
Observe that the assumption that $\forall \, n\in \N$,  $t \in [0, T] \colon \P(\mathbb{X}_t^n = \Y_t^n)=1$ yields that
\begin{equation}
\bigg\{\omega \in \Omega \colon  \big(\forall \, m \in \N, s \in  [0, T] \colon \mathbb{X}_{\lf s \rf_{h_m}}^m(\omega)= \Y_{\lf s \rf_{h_m}}^m(\omega) \big)\bigg\} \in \F
\end{equation}
and 
\begin{equation}
\P\!\left(\forall \, m \in \N, s \in  [0, T] \colon \mathbb{X}_{\lf s \rf_{h_m}}^m= \Y_{\lf s \rf_{h_m}}^m\right)=1.
\end{equation}
Combining this and \eqref{eq:Omega_1} demonstrates that
\begin{equation}
\label{eq:prob:Omega_1}
\Omega_1\in  \left\{ B\in \mathcal{F}\colon  \mathbb{P}(B)=1\right\}
.
\end{equation}
Next note that the fact that for all $r\in [0,1]$, $t\in [0,T]$ it holds that
$\|(t (\kappa-A))^r \, e^{t A}\|_{L(H)}\leq e^{t  \kappa}$ (cf., e.g., Lemma~11.36 in \citet{renardy2006introduction})
implies for all $\varepsilon \in [0,1-\alpha-\varrho]$ that  
\begin{align} 
\label{eq:sup_s_alpha_finite} 
\begin{split} 
&\sup\nolimits_{ s \in [0, T] } \big( s^{ ( \varrho + \varepsilon + \alpha ) } \| e^{ s A } \|_{ L( H_{ - \alpha }, H_{ \varrho + \varepsilon} ) } \big) \\
& = \sup\nolimits_{ s \in [0, T] } \| ( s(\kappa-A) )^{ ( \varrho + \varepsilon + \alpha ) } \, e^{ s A } \|_{ L(H) } \leq e^{T  \kappa} < \infty.
\end{split} 
\end{align} 
Therefore, we obtain for all $ n \in \N $, $ t \in [0, T] $, $ \varepsilon \in [0, 1 - \varrho - \alpha) $ that
\begin{align} 
\begin{split} 
&
	\int_0^t \| ( \Id_{H_{\varrho}} - P_n|_{H_{\varrho}} ) \, e^{ s A } \|_{ L( H_{ - \alpha } , H_{ \varrho } ) } \, ds 
\\
& 
\leq \int_0^t \| \Id_{H_{\varrho+\varepsilon}} - P_n|_{H_{\varrho+\varepsilon}} \|_{ L( H_{ \varrho + \varepsilon } , H_{ \varrho } ) } \, \| e^{ s A } \|_{ L( H_{ - \alpha }, H_{ \varrho + \varepsilon } ) } \, ds 
\\ 
& 
\leq  \| ( \kappa- A )^{ - \varepsilon } ( \Id_H - P_n ) \|_{ L( H ) } \int_0^t  e^{T  \kappa} \, s^{ - ( \varrho + \varepsilon + \alpha ) } \, ds \\ 
& 
=  \frac{  e^{T  \kappa} \,  \| ( \kappa - A )^{ - 1 } ( \Id_H - P_n ) \|_{ L( H ) }^{ \varepsilon }  \, t^{ ( 1 - \varrho - \varepsilon - \alpha ) } }{ ( 1 - \varrho -  \varepsilon -\alpha) } 
. 
\end{split} 
\end{align} 
This together with the assumption that $ \liminf_{ n \to \infty} \inf( \{\lambda_b \colon b \in \mathbb{H} \backslash \mathbb{H}_n \} \cup \{\infty\} ) = \infty $ proves that 
\begin{align}
\label{eq:sup_Id_H_0} 
	\limsup_{ n \to \infty } \left( \int_0^T \| (\Id_{H_{\varrho}} - P_n|_{H_{\varrho}}  ) e^{ s A } \|_{ L( H_{ - \alpha } , H_{ \varrho } ) } \, ds \right) \! = 0 . 
\end{align} 
Moreover, observe that the assumption that  $ \forall \, n\in \N$, $t\in [0,T]$, $\omega \in \Omega \colon
	\mathbb{O}^n_{t}(\omega)= \mathcal{O}^n_{t}(\omega) - \int_0^{t} e^{(t-s)(A-\eta)} \, \eta  \mathcal{O}^n_s(\omega) \, d s
$
and the fact that  $\forall \, t\in [0,T] \colon $ $ \|e^{ t A }\|_{L(H)} \leq e^{t \kappa}$ (cf., e.g., Lemma~11.36 in \citet{renardy2006introduction})
imply that for all $n\in \N$, $t\in [0,T]$, $\omega \in \Omega$ it holds that 
\begin{align} \label{eq:diff:O}
\begin{split}
	& \|\mathbb{O}^n_{t}(\omega) - \mathcal{O}^n_{t}(\omega) \|_{H_{\varrho}} 
	\\ 
	& 
	\leq  \smallint_0^{t} \|e^{(t-s)(A-\eta)}\|_{L(H)} \, \|\eta  \mathcal{O}^n_s (\omega) \|_{H_\varrho} \, d s
	\leq \smallint_0^t e^{(t-s) (\kappa-\eta)} \, \|\eta  \mathcal{O}^n_s(\omega) \|_{H_\varrho} \, d s\\
& 
	\leq  \,  \eta \, \smallint_0^T e^{(T-s) |\kappa-\eta|}  \, \| \mathcal{O}^n_s(\omega) \|_{H_\varrho} \, d s
	\leq \,  \eta \, T \, e^{T |\kappa-\eta|}  \left[ \sup\nolimits_{s\in [0,T]} \| \mathcal{O}^n_s(\omega) \|_{H_\varrho} \right] \! .
\end{split}
\end{align}
Therefore, we obtain for all $\omega\in \Omega$ that
\begin{align}
\label{eq:boldface:O}
\begin{split}
&
	\limsup_{n\to\infty} \sup_{t\in [0,T]} \|\mathbb{O}^{n}_{t}(\omega) \|_{H_\varrho}\\
& 
	\leq \limsup_{n\to\infty} \left( \sup_{t\in [0,T]} \|\mathbb{O}^{n}_t(\omega)  - \mathcal{O}^{n}_t(\omega) \|_{H_\varrho} 
+  \sup_{t\in [0,T]} \|\mathcal{O}^{n}_t(\omega)  \|_{H_\varrho} \right)\\
& 
	\leq \left(  \eta \, T \, e^{T |\kappa-\eta|}+1\right)  \limsup_{n\to\infty} \sup_{t\in [0,T]} \|\mathcal{O}^{n}_t (\omega) \|_{H_\varrho}
.
\end{split}
\end{align}
Furthermore, note that the assumption that $O \co [0,T]\times \Omega \to H_\varrho$ has continuous sample paths and \eqref{eq:O:convergence} ensure that for all $\omega \in \Omega_1$ it holds that
\begin{align}
\begin{split}
& 
	\limsup_{n\to\infty} \sup_{t\in [0,T]} \|\mathcal{O}^{n}_t (\omega) \|_{H_\varrho} \\
& 
	\leq \limsup_{n\to\infty} \sup_{t\in [0,T]} \|\mathcal{O}^{n}_t(\omega) -O_t (\omega)\|_{H_\varrho} +\sup_{t\in [0,T]} \|O_t (\omega) \|_{H_\varrho}\\
& 
	= \sup_{t\in [0,T]} \|O_t (\omega)\|_{H_\varrho}<\infty
.
\end{split}
\end{align}
Combining this with \eqref{eq:boldface:O} we obtain for all $\omega \in \Omega_1$ that
\begin{equation}
\limsup_{n\to\infty}\sup_{t\in [0,T]} \big\|\mathbb{O}^{n}_{ \lf t \rf_{h_{n}}}(\omega) \big\|_{H_{\varrho}} < \infty.
\end{equation}
The fact that $H_\varrho \subseteq H$ continuously hence shows that for all $\omega \in \Omega_1$ it holds that
\begin{equation}
\limsup_{n\to\infty}\sup_{t\in [0,T]} \big\|\mathbb{O}^{n}_{ \lf t \rf_{h_{n}}}
(\omega)\big\|_H < \infty.
\end{equation}
This, 
\eqref{eq:sup_Id_H_0}, 
 and Proposition~\ref{prop:euT:Xomega} 
(with 
$O=\left([0,T]\ni t \mapsto O_t(\omega) \in H_\varrho\right)$, 
$(\mathcal{O}^n)_{n\in \N}$ $=\left([0,T]\ni t \mapsto \mathcal{O}^{n}_t(\omega) \in H_\varrho\right)_{n\in \N}$,
$(\mathbb{O}^n)_{n\in \N}=\left([0,T]\ni t \mapsto \mathbb{O}^{n}_t(\omega) \in H_\varrho\right)_{n\in \N}$,
$(\Y^n)_{n\in \N}=\left([0,T]\ni t \mapsto {\bf X}^{n}_t(\omega) \in H_\varrho\right)_{n\in \N}$
for $\omega \in \Omega_1$
in the notation of Proposition~\ref{prop:euT:Xomega}) 
assure that for all $\omega \in \Omega_1$ it holds that
\begin{equation}
\label{eq:bf:X}
\liminf_{ n \to \infty} \sup_{t \in [0, T]} \| {\bf X}^{n}_t(\omega) \|_{H_{\varrho}} < \infty
\end{equation}
and that
there exists a unique function
$
Y(\omega)\in C([0,T],H_{\varrho})
$ 
which satisfies for all $t \in [0,T]$ that 
$
\int_0^t \|e^{(t-s)A} F(Y_s(\omega))\|_{H_{\varrho}} \, d s<\infty
$
and 
$
Y_t(\omega) = \int_0^t e^{ ( t - s ) A } \, F( Y_s (\omega)) \, d s + O_t(\omega)
$. 
Let $X \colon [0,T]\times \Omega \to H_\varrho$ be the function which satisfies for all $t\in [0,T], \omega\in \Omega$ that
\begin{align}
X_t(\omega) =
\begin{cases} 
	Y_t(\omega) & \colon \omega \in \Omega_1\\
	O_t(\omega) & \colon \omega \notin \Omega_1
\end{cases}.
\end{align}
Observe that for all $\omega \in \Omega$ it holds that
\begin{align} \label{item:continuity:path} 
X(\omega) \in C([0,T], H_{\varrho}).
\end{align} 
Moreover, note that for all  $t\in [0,T]$, $\omega\in\Omega_1$ it holds that
\begin{align} \label{item:solution:path} 
X_t(\omega) = \int_0^t e^{ ( t - s ) A } \, F( X_s(\omega) ) \, d s + O_t(\omega).
\end{align}
Furthermore, observe that \eqref{eq:EUN:Flip:3} proves that for all $r \in [0,\infty)$, $v, w \in H_{\varrho}$ satisfying $v\neq w$ and $\|v\|_{H_{\varrho}} + \|w\|_{\varrho} \leq r$ it holds that
\begin{align}
\label{eq:F:lipschitz}
\begin{split} 
\tfrac{\|F(v)-F(w)\|_{H_{-\alpha}}}{\|v-w\|_{H_{\varrho}}} 
& 
\leq 
\tfrac{\|v-w\|_{H_{\rho}}}{\|v-w\|_{H_{\varrho}}} \, \theta \, ( 1 + \| v \|_{ H_{ \rho } }^{ \vartheta } + \|w\|_{H_{\rho}}^{\vartheta}) 
\\
& \leq   \left[ \sup\nolimits_{u\in H_{\varrho}\setminus\{0\}} \tfrac{\|u\|_{H_{\rho}}}{\|u\|_{H_{\varrho}}} \right]  \theta \, (1+2 r^{\vartheta})  <\infty.
\end{split}
\end{align}
Combining this,
the fact that $ \limsup_{ n \to \infty } \big\| P_n |_{H_{\varrho}} \big\|_{ L( H_{ \varrho } ) } = 1 < \infty $,
\eqref{eq:sup_Id_H_0}, 
the assumption that $\limsup_{ n \to \infty} h_n =0 $, 
 \eqref{item:solution:path},
 \eqref{eq:tilde_X}, 
 and
\eqref{eq:bf:X} allows us to apply Lemma~\ref{lem:convergence1} 
(with 
$(V,  \left\| \cdot \right\|_V) = (H_{\varrho}, \left\| \cdot \right\|_{H_{\varrho}})$, 
$(W,  \left\| \cdot \right\|_W) = (H_{-\alpha}, \left\| \cdot \right\|_{H_{-\alpha}})$, 
 $ T = T $, 
 $\chi = \chi$,
$ J=[0,T]$,
 $ F = F $, 
$ S = \big((0,T] \ni t \mapsto ( H_{ - \alpha } \ni v \mapsto e^{ t A } v \in H_{ \varrho } ) \in L(H_{-\alpha},H_{\varrho})\big) $, 
$ \alpha = \varrho + \alpha $, 
$ (P_n)_{n \in \N} = ( H_{ \varrho } \ni v \mapsto P_n( v ) \in H_{ \varrho } )_{n \in \N} $, 
$O=\left([0,T]\ni t \mapsto O_t(\omega) \in H_\varrho\right)$, 
$(\mathcal{O}^n)_{n \in \N} =\left([0,T]\ni t \mapsto \mathcal{O}^{n}_t(\omega) \in H_\varrho\right)_{n\in \N}$,
$ (h_n)_{n \in \N} = (h_n)_{n \in \N} $, 
$X=([0,T]\ni t \mapsto X_t(\omega)$ $\in H_\varrho)$, 
$(\mathcal{X}_n)_{n \in \N} = ([0,T]\ni t \mapsto {\bf X}^{n}_t(\omega) \in H_\varrho)_{n\in \N}$ 
for $\omega \in \Omega_1$ in the notation of Lemma~\ref{lem:convergence1}) to obtain  
for all $\omega \in \Omega_1$ that
\begin{align} \label{eq:strong:pathconv}
\limsup_{n\to \infty}\sup_{t\in [0,T]} \|{\bf X}^{n}_t(\omega) - X_t(\omega)\|_{H_{\varrho}}=0.
\end{align}
This, in particular, implies that for all $t\in [0,T]$, $\omega\in \Omega_1$ it holds that
\begin{align}
\label{eq:pathwise:limit}
\limsup_{n\to\infty} \|{\bf X}^{{n}}_t(\omega) - X_t(\omega)\|_{H_{\varrho}}=0.
\end{align}
Moreover, note that Lemma~2.3 in \citet{hutzenthaler2016strong} and the assumption that $\mathcal{O}^n \colon [0,T]\times \Omega \to H_\varrho$, $n \in \N$, are stochastic processes with continuous sample paths   ensure that  ${\bf X}^n:[0,T] \times \Omega \to H_{\varrho}$, $n \in \N$,  are stochastic processes with right-continuous sample paths. This, \eqref{eq:pathwise:limit}, the fact that $\forall \, t \in [0, T] \colon X_t|_{\Omega\setminus\Omega_1}=O_t|_{\Omega\setminus\Omega_1}$, 
and the fact that $\Omega_1 \in \mathcal{F}$ 
prove that $X \colon [0,T]\times \Omega \to H_{\varrho}$ is a stochastic process.
Combining this, the fact that $\P(\Omega_1)=1$, \eqref{item:continuity:path}, and \eqref{item:solution:path}
ensures $X \colon [0,T]\times \Omega \to H_\varrho$ is a stochastic process with continuous sample paths which  satisfies that for all $t\in [0,T]$  it holds $\P$-a.s.~that
\begin{align}
\label{eq:X_existence}
X_t = \int_0^t e^{ ( t - s ) A } \, F( X_s ) \, d s + O_t.
\end{align}
In the next step let $Z \colon [0, T] \times \Omega \to H_\varrho$ be another stochastic process with continuous sample paths
 which  satisfies that for all $t\in [0,T]$ it holds $\P$-a.s.~that 
$
 Z_t = \int_0^t e^{ ( t - s ) A } \, F( Z_s ) \, d s + O_t
$.
This ensures that there exists an event $\Omega_2 \in \{ B\in \mathcal{F}\colon  \mathbb{P}(B)=1\}$ such that for all $t \in [0, T]$, $\omega \in \Omega_2$ it holds that
 \begin{align}
 Z_t(\omega) = \int_0^t e^{ ( t - s ) A } \, F( Z_s(\omega) ) \, d s + O_t(\omega).
 \end{align}
Combining this, \eqref{eq:F:lipschitz}, \eqref{item:solution:path}, 
\eqref{eq:sup_s_alpha_finite}, and, e.g.,  Corollary~6.1 in \citet{jentzen2018existence} (with 
$(V,  \left\| \cdot \right\|_V) = (H_{\varrho}, \left\| \cdot \right\|_{H_{\varrho}})$, 
$(W,  \left\| \cdot \right\|_W) = (H_{-\alpha}, \left\| \cdot \right\|_{H_{-\alpha}})$, 
$T=T$,
$\tau=T$,
$F=F$,
$x^1=([0,T]\ni t \mapsto X_t(\omega) \in H_\varrho)$,
$x^2=([0,T]\ni t \mapsto Z_t(\omega) \in H_\varrho)$,
$ o = ([0,T]\ni t \mapsto O_t(\omega) \in H_\varrho)$,
$S=\big((0,T)\ni s \mapsto e^{s A} \in L(H_{-\alpha}, H_{\varrho})\big)$
for $\omega\in \Omega_1 \cap \Omega_2$ in the notation of Corollary~6.1 in \citet{jentzen2018existence})
demonstrates that for all $t\in [0,T]$, $\omega\in \Omega_1 \cap \Omega_2$ it holds that $X_t(\omega)=Z_t(\omega)$. This and the fact that $\Omega_1 \cap \Omega_2 \in  \{ B\in \mathcal{F}\colon  \mathbb{P}(B)=1\}$
show that the stochastic processes $X$ and $Z$ are indistinguishable.
This and \eqref{eq:X_existence} establish Item~\eqref{item:solution}.
In the next step we combine \eqref{eq:tilde_X},  \eqref{eq:strong:pathconv}, and the fact that $\forall \, n \in \N, t\in [0,T],  \omega \in \Omega_1 \colon {\bf X}^n_t(\omega)  =   \mathbb{X}_t^n(\omega)$ to obtain  that for all $\omega\in \Omega_1$ it holds that
\begin{align}
\limsup_{n \to \infty} \sup_{t \in [0, T]} \| X_t(\omega)- \mathbb{X}_t^{n}(\omega) \|_{H_\varrho}=0.
\end{align} 
This and \eqref{eq:prob:Omega_1} establish Item~\eqref{item:convergence:prob}. 
The proof of Theorem~\ref{thrm:EU} is thus completed.
\end{proof}

\begin{cor}[Strong convergence]\label{cor:strong}
Assume Setting~\ref{section:EUN:main},
let $p\in [2,\infty)$, and assume that $\limsup_{ n \to \infty} \sup_{ t \in [0,T]} \E[ \| \mathbb{O}_t^n \|_{H}^p] < \infty$,
\begin{align}
\label{eq:convergence_O}
\limsup_{n \to \infty}  \E \Big[ \!\min\!\left\{ 1, \sup\nolimits_{t \in [0,T]}\| O_t- \mathcal{O}_t^{n} \|_{ H_{ \varrho } } \right\}\! \Big] \! = 0,
\end{align}
and
\begin{multline} \label{eq:strong:hp_int}
\limsup_{ n \to \infty}  \E \bigg[ \int_0^T e^{ \int_r^T p\, \phi( \mathbb{O}_{\lf u \rf_{h_n} }^n )  \, du} \Big( 1+ \big|\Phi\big(\mathbb{O}_{ \lf r \rf_{h_n} }^n \big)\big|^{\frac{p}{2}} + \|\mathbb{O}_r^n\|_H^{p} \\
 + \smallint\nolimits_{0}^T \| \mathcal{O}_u^n \|_{H_{\varrho}}^{2p+ 2p\vartheta} \, du \Big) \, dr\bigg] < \infty.
\end{multline}
Then 
\begin{enumerate}[(i)]
	\item \label{item:cor:solution} there exists an up to indistinguishability unique stochastic process $X  \colon$ $ [0, T] \times \Omega \to H_\varrho$ with continuous sample paths which satisfies that for all $t\in [0,T]$ it holds $\P$-a.s.~that
	\begin{align} \label{cor:strong:X}
X_t = \int_0^t e^{ ( t - s ) A } \, F( X_s ) \, ds + O_t,
	\end{align}
	\item \label{item:exp}  it holds that $\mathbb{X}^n \colon [0,T] \times \Omega \to H_{\varrho} $, $n\in \N$, are stochastic processes with right-continuous sample paths and \begin{align}
	\limsup_{n \to \infty}  \E \Big[\!\min\!\left\{ 1, \sup\nolimits_{t \in [0,T]}\| X_t- \mathbb{X}_t^{n} \|_{ H_{ \varrho } } \right\} \Big] \! = 0,
	\end{align} 
	\item \label{item:pnorm} it holds that $ \limsup_{ n \to \infty } \sup_{ t \in [0,T] } \E\big[ \| X_t \|^p_{H} + \| \Y^n_t \|_{H}^p \big] < \infty $, and
	\item \label{item:strong} it holds for all $ q \in (0, p) $ that $ \limsup_{ n \to \infty } \sup_{ t \in [0,T] } \E\big[ \| X_t - \Y_t^n \|_{H}^q \big] = 0 $.
\end{enumerate}
\end{cor}
\begin{proof}[Proof of Corollary~\ref{cor:strong}]
First, note that \eqref{eq:convergence_O} implies there exists a strictly increasing function $k \colon \N \to \N$ such that
\begin{align}
\sum_{n=1}^{\infty}  \E \Big[ \! \min \big\{ 1, \sup\nolimits_{t \in [0,T]}\| O_t- \mathcal{O}_t^{k(n)} \|_{ H_{ \varrho } } \big\} \Big] < \infty.
\end{align}
Lemma~3.1 in \citet{jentzen2019strong} (with $(\Omega, \mathcal{F}, \P)= (\Omega, \mathcal{F}, \P)$, $E=C([0,T], H_{\varrho})$, $d= ( C([0,T], H_{\varrho}) \times C([0,T], H_{\varrho}) \ni (x,y) \mapsto \sup_{t \in [0,T]} \|x(t)-y(t)\|_{H_{\varrho}} \in [0,\infty) )$, $(X_n)_{n \in \N} = (\mathcal{O}^{k(n)})_{n \in \N}$, $X_0 = O$  in the notation of Lemma~3.1 in \citet{jentzen2019strong}) hence proves that 
\begin{align}
\label{eq:cor:strong:limsup:O}
\P\! \left(\limsup\nolimits_{n \to \infty} \sup\nolimits_{t \in [0,T]}\| O_t- \mathcal{O}_t^{k(n)} \|_{ H_{ \varrho } } = 0\right)\! =1.
\end{align}
Next observe that \eqref{eq:strong:hp_int} implies that
\begin{multline} 
\limsup_{ n \to \infty}  \E \bigg[ \int_0^T e^{ \int_r^T p\, \phi( \mathbb{O}_{\lf u \rf_{h_{k(n)}} }^{k(n)} )  \, du} \Big( 1+ \big|\Phi\big(\mathbb{O}_{ \lf r \rf_{h_{k(n)}} }^{k(n)} \big)\big|^{\frac{p}{2}} + \|\mathbb{O}_r^{k(n)}\|_H^{p} \\
+ \smallint\nolimits_{0}^T \| \mathcal{O}_u^{k(n)} \|_{H_{\varrho}}^{2p+ 2p\vartheta} \, du \Big) \, dr\bigg] < \infty.
\end{multline}
This, in particular, yields that
\begin{multline} \label{eq:cor:hp_int_bis}
 \P\bigg(\liminf\nolimits_{ n \to \infty}  \int_0^T e^{ \int_r^T 2\, \phi( \mathbb{O}_{\lf u \rf_{h_{k(n)}} }^{k(n)} )  \, du} \Big[ 1+ |\Phi(\mathbb{O}_{ \lf r \rf_{h_{k(n)}} }^{k(n)} )| + \|\mathbb{O}_r^{k(n)}\|_H^{2}  \\
 + \smallint \nolimits_{0}^T \| \mathcal{O}_u^{k(n)} \|_{H_{\varrho}}^{4+ 4\vartheta} \, du \Big] \, dr< \infty\bigg)=1.
\end{multline}
Combining this with \eqref{eq:cor:strong:limsup:O} and Item~\eqref{item:solution} in Theorem~\ref{thrm:EU} (with $P_n = P_{k(n)}$, $\mathbb{H}_n = \mathbb{H}_{k(n)}$, $h_n = h_{k(n)}$, $\mathcal{X}^n = \mathcal{X}^{k(n)}$, $\mathcal{O}^n = \mathcal{O}^{k(n)}$, $O=O$, $\mathbb{X}^n = \mathbb{X}^{k(n)}$, and $\mathbb{O}^n = \mathbb{O}^{k(n)}$ for $n \in \N$ in the notation of Theorem~\ref{thrm:EU}) assures that there exists an up-to-indistinguishability unique stochastic process $X  \colon$ $ [0, T] \times \Omega \to H_\varrho$ with continuous sample paths which satisfies that for all $t\in [0,T]$ it holds $\P$-a.s.~that
\begin{align} 
X_t = \int_0^t e^{ ( t - s ) A } \, F( X_s ) \, ds + O_t.
\end{align}
This establishes Item~\eqref{item:cor:solution}. 
Next note that
the assumption that $ \Y^n, \mathcal{O}^n \colon [0, T] \times \Omega \to H_\varrho$, $ n \in \N$, are stochastic processes and \eqref{boldface_X} prove that for all $n \in \N$ it holds that 
$\mathbb{X}^n \co [0,T]\times \Omega \to H_\varrho$ is also a stochastic process.
The assumption that $\mathcal{O}^n \co [0,T] \times \Omega \to H_{\varrho}$, $n\in\N,$ are continuous, and, e.g., Lemma~2.2 in \citet{hutzenthaler2016strong}  therefore ensure that  $\mathbb{X}^n \co [0,T] \times \Omega \to H_{\varrho}$, $n\in \N$, are stochastic processes with right-continuous sample paths.
Next observe that the fact that $H\subseteq H_{-1}=\bar{ H}^{H_{-1}}$ and the fact that for all $n\in \N$ it holds that $P_n\in L(H)$ imply that there exist $\tilde P_n \in L(H_{-1},H), n\in \N,$ such that for all $v\in H, n\in \N$ it holds that $\tilde P_n(v) = P_n(v)$.
Items~(i),(ii) and (iii) in Theorem~3.5 in~\citet{jentzen2019strong} (with 
 $ P_n=\tilde P_{n}$,  
 $n\in \N$,
in the notation of Theorem~3.5 in~\citet{jentzen2019strong}) therefore establish Items~\eqref{item:exp},\eqref{item:pnorm}, and \eqref{item:strong}. 
The proof of Corollary~\ref{cor:strong} is thus completed.
\end{proof}

\section{Example: Stochastic Burgers equations} \label{sec:burgers}

In this section we apply Corollary~\ref{cor:strong} to  the stochastic Burgers equations with space-time white noise.
Throughout this section we use the following notation. 
For a set $A \in \mathcal{B}(\R)$ we denote by $\lambda_A \colon \mathcal{B}(A) \to [0, \infty]$ the Lebesgue-Borel measure on $A$. 
For a measure space $(\Omega, \F, \mu)$, a measurable space $(S, \mathcal{S})$, a set $R$, and a function $f \colon \Omega \to R$ we denote by $\left[ f \right]_{\mu, \mathcal{S}} $ the set given by 
\begin{align}
\begin{split}
\left[f \right]_{\mu, \mathcal{S}} 
&=
\left\{ g \colon \Omega \to S \colon  g \text{ is } \F\slash \mathcal{S}\text{-measurable and } \right.
\\
	&  \qquad \left. \exists \, A \in \F \colon \mu(A)=0 \,\, \text{and} \,\, \{ \omega \in \Omega \colon f(\omega) \neq g(\omega)\} \subseteq A \right\}. 
\end{split}
\end{align}
We denote by 
$
(\und{\cdot}) \colon\ \{ [v]_{\lambda_{(0,1)}, \mathcal{B}(\R)} \in L^0(\lambda_{(0,1)}; \R) \colon v \in C( (0,1), \R ) \} \to C((0,1), \R)
$  
the function which satisfies for all $v \in C( (0,1), \R ) $  that
$
\und{[v]_{\lambda_{(0,1)}, \mathcal{B}(\R)}}=v
$.

\begin{sett} \label{setting:example:Burgers}
Let
$ c_1\in \R$,
$ T, c_0 \in (0,\infty)$,
$  \kappa \in [0, \infty)$,
$\alpha=\nicefrac{1}{2}$,
$\rho=\nicefrac{1}{8}$,
$\gamma\in (\nicefrac{3}{4},\nicefrac{7}{8})$,
$\varrho\in (\nicefrac{1}{8},1-\gamma)$,
$ \vartheta =1 $,
$ \chi \in  (0, \nicefrac{( \varrho - \rho ) }{( 1 + \vartheta) }] $,
let 
\begin{align}
( H, \left< \cdot , \cdot \right>_H, \left\| \cdot \right\|_H ) = (L^2(\lambda_{(0,1)}; \R), \langle \cdot , \cdot \rangle_{L^2(\lambda_{(0,1)}; \R)}, \left\| \cdot \right\|_{L^2(\lambda_{(0,1)}; \R)} ),
\end{align}
let
$\left(e_n\right)_{n\in \N} \subseteq H$ and $\left(\lambda_n \right)_{n\in \N} \subseteq (0, \infty)$
satisfy for all $ n \in \N $ that
\begin{align}
e_n = [ ( \sqrt{2} \sin(n \pi x) )_{x \in (0,1)}]_{\lambda_{(0,1)} , \mathcal{B}(\R)} \quad \text{ and } \quad \lambda_n = c_0 \pi^2 n^2,
\end{align}
let
$ A \colon D(A) \subseteq H \to H $
be the linear operator which satisfies
$ D(A) = \{ v \in H \colon \sum_{k \in \N} | \lambda_k \langle e_k , v \rangle_H |^2 < \infty \} $
and
$ \forall \, v \in D(A) \colon A v = \sum_{k \in \N} - \lambda_k \langle e_k , v \rangle_H \, e_k$,
let
$ ( H_r, \left< \cdot , \cdot \right>_{ H_r }, \left\| \cdot \right\|_{ H_r } ) $, $ r \in \R $,
be a family of interpolation spaces associated to $ \kappa - A $ (see, e.g., \citet[Section~3.7]{sell2002dynamics}),
let $F \colon H_{1/8} \to H_{-1/2}$ satisfy for all $v \in H_{1/8}$ that 
\begin{align}\label{setting:F}
F(v)= c_1(v^2)',
\end{align}
let $\left(P_n \right)_{n\in\N} \subseteq L(H)$ 
satisfy
for all $ u \in H $, $ n \in \N $ that
\begin{align} 
P_n(u) = \sum_{k =1 }^n \langle e_k, u \rangle_H \, e_k ,
\end{align} 
let $\left( h_n \right)_{n\in\N} \subseteq (0, T]$ satisfy that $ \limsup_{ m \to \infty} h_m =0$,
let $\xi \in H_{\nicefrac{1}{2}} $,
let $ ( \Omega, \F, \P ) $ be a probability space  with a normal filtration $(\mathbb{F}_{t})_{t \in [0,T]}$,
let $(W_t)_{t \in [0, T]}$ be an $\Id_H$-cylindrical $( \Omega, \F, \P, (\mathbb{F}_{t})_{t \in [0,T]} )$-Wiener process,
let
$ \Y^n, \mathcal{O}^n, \Psi^n \colon [0, T] \times \Omega \to P_n(H)$, $ n \in \N$,
be stochastic processes
which satisfy
for all $ n \in \N $, $ t \in [ 0, T ] $ that
$
\Psi^n_t = P_n \, e^{ t A } \, \xi + \mathcal{O}_t^n,
$
\begin{align}
\left[\mathcal{O}_t^n \right]_{\P, \mathcal{B}(H)} = \int_0^t P_n \, e^{(t-s)A}  \, dW_s,
\end{align}
and
\begin{equation}\label{eq:set:abstract}
\P \! \left( \Y_t^n = \Psi^n_t + \smallint_0^t P_n \,  e^{  ( t - s ) A } \, \one_{ \{ \| \Y_{ \lf s \rf_{h_n} }^n \|_{ H_{\varrho} } + \| \Psi^n_{ \lf s \rf_{h_n} }\|_{ H_{\varrho} } \leq | h_n|^{ - \chi } \}} \, F \big(  \Y_{ \lf s \rf_{ h_n } }^n \big) \, ds  \right)\! =1. 
\end{equation}
\end{sett}

\subsection{Properties of the nonlinearity}

The following lemma shows that the function $F$ in Setting~\ref{setting:example:Burgers} above satisfies the elementary property \eqref{eq:add:property}.

\begin{lemma}
\label{lem:normHr:Burgers}
Assume Setting~\ref{setting:example:Burgers} and let $r\in (\nicefrac34,\infty)$. 
Then it holds for all $v\in H_{\nicefrac12}$ that $F(v)\in H$ and that
\begin{align}
\| F(v)\|_{H_{-r}} \leq  |c_1|  
\left(\sum_{n\in\N} 2 \pi^2 n^2 \, (\kappa+ c_0 \, \pi^2\, n^2 )^{-2r}  \right
)^{\!\! \nicefrac12}  \|v\|^2_{H} < \infty.
\end{align}
\end{lemma}

\begin{proof}[Proof of Lemma~\ref{lem:normHr:Burgers}]
Throughout the proof let $v\in H_{\nicefrac12}$. Observe that, e.g., Lemma~4.5 in \citet{jentzen2016exponential} ensures that $F(v)\in H$.  
Hence, we obtain that
\begin{align}\label{eq:normHgamma}
\begin{split}
\| F(v)\|_{H_{-r}} 
& =
\sup\nolimits_{u \in H\setminus\{0\}} \tfrac{ \left| \langle  F (v) ,   (\kappa-A)^{-r}\, u \rangle_H \right|}{\|u\|_H} \\
& = |c_1|  \left[
\sup\nolimits_{u \in H\setminus\{0\}} \tfrac{ \left| \langle  (v^2)' ,  (\kappa-A)^{-r}\,u \rangle_H \right|}{\|u\|_{H}} \right] \! .
\end{split}
\end{align}
Next note that for all $u\in H$ it holds that $(\kappa-A)^{-r} \, u \in H_r$
and  
\begin{align}
\begin{split}
 (\kappa-A)^{-r} \, u & = \smallsum_{n\in\N} (\kappa+ \lambda_n)^{-r} \, \langle  u , e_n \rangle_H \, e_n.
\end{split}
\end{align}
This ensures that for all $u\in H$ it holds that
\begin{align*}
 \left( (\kappa-A)^{-r}\, u \right)'  &= \smallsum_{n\in\N}  (\kappa+ \lambda_n)^{-r}  \, \langle  u , e_n \rangle_H \, e'_n \numberthis\\
& =  \smallsum_{n\in\N}  (\kappa+ c_0 \, \pi^2\, n^2 )^{-r}  \, \langle  u , e_n \rangle_H \, \left[\left( \sqrt{2} \pi n \, \cos(\pi n x)  \right)_{x \in (0,1)}\right]_{\lambda_{(0,1)} , \mathcal{B}(\R)} \!\! .
\end{align*}
The Cauchy-Schwarz inequality 
hence imply that for all $u\in H$ it holds that
\begin{align*}
\left\|\left( (\kappa-A)^{-r}\, u \right)' \right\|_{L^{\infty}(\lambda_{(0,1)};\R)}
	& \leq 
\smallsum_{n\in\N} \sqrt{2} \pi n \,(\kappa+ c_0 \, \pi^2\, n^2 )^{-r}  \, \left| \langle  u , e_n \rangle_H \right|
	\\
	& \leq 
\left(\smallsum_{n\in\N} 2 \pi^2 n^2 \, (\kappa+ c_0 \, \pi^2\, n^2 )^{-2 r} \right)^{\! \nicefrac12} 
\! 
\left(\smallsum_{n\in \N} \langle e_n, u\rangle_H^2 \right)^{\! \nicefrac12}
	\\
	& = 
\left(\smallsum_{n\in\N} 2 \pi^2 n^2 \, (\kappa+ c_0 \, \pi^2\, n^2 )^{-2r}  \right)^{\!\nicefrac12} \! \|u\|_H. \numberthis
\end{align*}
Combining this with \eqref{eq:normHgamma} yields that
\begin{align*}\label{eq:normHgamma:2}
 \| F(v)\|_{H_{-r}} 
	& 
\leq  |c_1| \, \left[
\sup\nolimits_{u \in H\setminus\{0\}} \tfrac{ \left| \langle  v^2 , \left(  (\kappa-A)^{-r} \, u \right)' \rangle_H \right|}{\|u\|_{H}} \right] 
	\\
	& 
\leq |c_1|  \left[
\sup\nolimits_{u \in H\setminus\{0\}} \tfrac{ \|u\|_H }{\|u\|_{H}} \right] \left(\smallsum_{n\in\N} 2 \pi^2 n^2 \, (\kappa+ c_0 \, \pi^2\, n^2 )^{-2r}  \right)^{\!\nicefrac12} \|v^2\|_{L^1(\lambda_{(0,1)}; \R)}\\
& = |c_1|  \left(\smallsum_{n\in\N} 2 \pi^2 n^2 \, (\kappa+ c_0 \, \pi^2\, n^2 )^{-2r}  \right)^{\!\nicefrac12} \|v\|^2_H. \numberthis
\end{align*}
Moreover, observe that the fact that $r>\nicefrac{3}{4}$ assures that
$\sum_{n\in\N} 2 \pi^2 n^2 \, (\kappa+ c_0 \, \pi^2\, n^2 )^{-2r}  < \infty$.
The proof of Lemma~\ref{lem:normHr:Burgers} is thus completed.
\end{proof}


\subsection{Existence and uniqueness of the solution and strong convergence of the approximation scheme}
\begin{cor} \label{cor:strong:burgers}
Assume Setting~\ref{setting:example:Burgers} and let $p\in (0,\infty)$.
Then there exists a unique stochastic process $ X\colon [0, T] \times \Omega \to H_{\varrho}$ with continuous sample paths which satisfies for all $t \in [0, T]$  that  
\begin{align} \label{cor:strong:burgers:bochner}
\P\left( \int_0^t  \left\|e^{  ( t - s ) A}  \,F (  X_s )\right\|_{H_{\varrho}} \, ds <\infty \right)\!=1,
\end{align}
that
\begin{align} \label{cor:strong:burgers:X}
[X_t ]_{\P, \mathcal{B}(H)}  =    \left[ e^{ t A }   \xi  + \smallint_0^t e^{  ( t - s ) A}  \, F (  X_s ) \, ds \right]_{\P, \mathcal{B}(H)} \!\! + \int_0^t e^{(t-s)A} \, dW_s,
\end{align}
and that
\begin{align} \label{cor:strong:burgers:strong}
\limsup_{n \to \infty} \sup_{t \in [0,T]} \E \big[ \| X_t -\Y_t^n \|_H^p \big] \! = 0.
\end{align}
\end{cor}
\begin{proof}[Proof of Corollary~\ref{cor:strong:burgers}]
Throughout this proof let
$q\in [\max\{2,p\},\infty)$ be a real number, let $\theta, c \in (0,\infty)$ be the real numbers given by 
\begin{equation}
\theta =  |c_1|  |c_0|^{-\nicefrac{1}{2}}  \bigg[\sup\nolimits_{u \in  H_{ 1/8 }\backslash \{0\}}  \tfrac{\|u\|^2_{L^4(\lambda_{(0,1)}; \R)}}{\|(-A)^{\nicefrac{1}{8}} u\|^2_H} \bigg]+1
\end{equation}
and 
$
c=|c_1| \,  \left(\sum_{n\in\N} 2 \pi^2 n^2 \, (\kappa+ c_0 \, \pi^2\, n^2 )^{-2\gamma}  \right)^{\!\nicefrac12}
$, 
and let $ \phi , \Phi \colon H_1 \to [0,\infty)$ be the functions which satisfy for all $ v \in H_{1}$  that  
\begin{equation}
\phi(v) =  \max\left\{\tfrac{2|c_1|^2 }{c_0} , 4\right\}  \left[1+\sup\nolimits_{x \in (0,1)} |\und{v}(x)|^2 \right]
\end{equation}
and 
\begin{equation}
\Phi(v)= \max\left\{\tfrac{2|c_1|^2 }{c_0} , 4\right\} \left[1+ \sup\nolimits_{x \in (0,1)} |\und{v}(x)|^{ \max\{\nicefrac{2|c_1|^2 }{c_0} , 4\}}\right]
\! .
\end{equation}
Then note that
Lemma~6.3 in \citet{jentzen2019strong} shows that for all $v,w\in H_{\nicefrac12}$ it holds that $F(v+w)\in H$ and that
\begin{align} \label{burgers:cor:coer}
\begin{split}
& \left< v, F( v + w ) \right>_H 
	\\
&\leq \max\big\{\tfrac{2|c_1|^2 }{c_0} , 4\big\} \| v \|_H^2 \big[\!\sup\nolimits_{x \in (0,1)} |\und{w}(x)|^2\big]+ \tfrac{3}{4}\| (-A)^{\nicefrac{1}{2}}v \|^2_{ H } 
	\\
& \quad + \max\big\{\tfrac{2|c_1|^2 }{c_0} , 4\big\} \left[1+ \!\sup\nolimits_{x \in (0,1)} |\und{w}(x)|^{ \max\{\nicefrac{2|c_1|^2 }{c_0} , 4\}}\right]
	\\
& \leq \phi(w) \| v \|^2_H+ \tfrac{3}{4} \| (-A)^{\nicefrac{1}{2}} v \|^2_{ H }+ \Phi( w ).
\end{split}
\end{align}
Moreover, observe that Lemma~6.4 in \citet{jentzen2019strong} demonstrates that for all $v, w \in H_{\nicefrac{1}{8}}$ it holds that
\begin{align} \label{burgers:cor:con}
\begin{split}
&\|F(v) - F(w) \|_{H_{-1 /2}}
	 \leq 
|c_1|  |c_0|^{-\nicefrac{1}{2}} \left[\sup\nolimits_{u \in  H_{ 1/8 } \setminus \{0\}} \tfrac{\|u\|^2_{L^4(\lambda_{(0,1)}; \R)}}{\|(-A)^{\nicefrac{1}{8}} u\|^2_H} \right] 
	\\
	& \quad \cdot
\big( 1+ \|(-A)^{\nicefrac{1}{8}} v \|_{ H} +  \| (-A)^{\nicefrac{1}{8}} w \|_{ H} \big) \|(-A)^{\nicefrac{1}{8}}( v - w) \|_{ H}\\
& \leq \theta \big( 1+ \left\| v \right\|_{ H_{1/8 }} +  \left\| w \right\|_{ H_{1/8 }} \big) \left\| v - w \right\|_{ H_{1/8 }}.
\end{split}
\end{align}	
Furthermore, note that Lemma~\ref{lem:normHr:Burgers} assures that for all $v\in H_{\nicefrac12}$ it holds that $F(v)\in H$ and 
\begin{align}
\label{eq:F:gamma}
\| F(v) \|_{H_{-\gamma}} \leq c \|v\|^2_H.
\end{align}
In addition, observe that Lemma~5.6 in \citet{jentzen2019strong} (with $p=q$ in the notation of Lemma~5.6 in \citet{jentzen2019strong}) proves that there exist a real number $\eta \in [0,\infty)$ and stochastic processes $O \colon [0, T] \times \Omega \to H_{\varrho}$, ${\mathcal{Q}}^n, {\mathfrak{Q}}^n \colon [0, T] \times \Omega \to P_n(H)$, $n \in \N$, with continuous sample paths  which satisfy for all $n \in \N$, $t \in [0, T]$ that 
$
[O_t]_{\P, \mathcal{B}(H)} =  \int_0^{t}  e^{(t-s)A} \,  dW_s
$,
$
[{\mathcal{Q}}^n_t]_{\P, \mathcal{B}(H)} =  \int_0^{t} P_n \, e^{(t-s)A} \, dW_s
$,
$
{\mathfrak{Q}}^n_t= {\mathcal{Q}}_t^n + P_n  e^{t A} \xi- \int_0^t e^{(t-s)(A-\eta)} \, \eta \left( {\mathcal{Q}}_s^n + P_n  e^{s A} \xi\right) ds
$,
\begin{align}\label{eq:O:as}
\P \bigg( \limsup_{m \to \infty} \sup_{s \in [0, T]} \big\| (O_s + e^{s A}  \xi) - ({\mathcal{Q}}_s^m + P_m  e^{sA}  \xi )\big\|_{H_{\varrho}} \!=0 \bigg) \! =1,
\end{align}
that
\begin{multline}\label{eq:strong:limsup}
\limsup_{ m \to \infty} 
\bigg( \E\bigg[ \int_0^T 
\exp \left( \smallint_r^T q\,\phi\big( {\mathfrak{Q}}_{\lf u \rf_{h_m} }^m \big) \, du \right) \, 
\max\Big\{ \big|\Phi({\mathfrak{Q}}_{ \lf r \rf_{h_m} }^m )\big|^{q/2},1, \big\|{\mathfrak{Q}}_r^m\big\|_H^q,  \\
\smallint\nolimits_{0}^T \big\| {\mathcal{Q}}_u^m+ P_m  e^{u A} \xi \big\|_{H_{\varrho}}^{2q+ 2q\vartheta} \, du \, \Big\} \, dr  \bigg] + \sup_{ s \in [0,T]} \E[ \| {\mathfrak{Q}}_s^m \|_H^q]  \bigg) < \infty,
\end{multline}
and that
\begin{multline}\label{eq:X^n}
\P \Big( \Y_t^n = \smallint_0^t P_n \,  e^{  ( t - s ) A } \, \one_{ \{ \| \Y_{ \lf s \rf_{h_n} }^n \|_{ H_{\varrho} } + \| {\mathcal{Q}}_{ \lf s \rf_{h_n} }^n +P_n  e^{ \lf s \rf_{ h_n } A } \xi \|_{ H_{\varrho} } \leq | h_n|^{ - \chi } \}} \, F \big(  \Y_{ \lf s \rf_{ h_n } }^n \big) \, ds \\
+ P_n  e^{ t A } \xi  + {\mathcal{Q}}_t^n \Big)=1.
\end{multline}
Combining this with \eqref{burgers:cor:coer}--\eqref{eq:F:gamma}, as well as
Item~\eqref{item:cor:solution} and Item~\eqref{item:strong} in
Corollary~\ref{cor:strong} 
(with
 $ \mathbb{H} = \{ e_k \in H \colon k \in \mathbb{N} \} $,
 $\mathbb{H}_n = \{ e_k \in \mathbb{H} \colon k \in \{1, \ldots , n-1, n\}\}$, 
 $ \vartheta = 1$, 
 $ \epsilon=0$,
 $ \varphi=  \nicefrac{3}{4}$, 
 $ \alpha =\nicefrac{1}{2} $,
 $ \rho= \nicefrac{1}{8}$, $\varrho=\varrho$,
 $\mathcal{O}^n = ( [0,T] \times \Omega \ni (t,\omega) \mapsto ({\mathcal{Q}}^n_t( \omega ) + P_n  e^{t A} \xi) \in H_{\varrho} ) 
$,     
 $ O = ( [0,T] \times \Omega \ni (t,\omega) \mapsto (O_t( \omega ) + e^{tA}  \xi) \in H_{\varrho}) $,
 $ \mathbb{O}^n = ( [0,T] \times \Omega \ni (t,\omega) \mapsto {\mathfrak{Q}}^n_t( \omega )  \in H_{\varrho} ) $,
 $n\in \N$, 
 $ p=q$
 in the notation of Corollary~\ref{cor:strong}) 
we obtain  that for all $t\in [0,T]$ equations \eqref{cor:strong:burgers:bochner} and \eqref{cor:strong:burgers:X} hold and that for all $u\in (0, q)$ it holds that
\begin{align}
\limsup_{n \to \infty} \sup_{t \in [0,T]} \E \big[ \| X_t -\Y_t^n \|_H^{u} \big] = 0.
\end{align}
This, in particular,  establishes \eqref{cor:strong:burgers:strong}.
The proof of Corollary~\ref{cor:strong:burgers} is thus completed.
\end{proof}

{\bf Acknowledgments:} 
The authors would like to warmly thank Arnulf Jentzen for fruitful discussions, guidance, and support. 
This project has been supported by the Deutsche Forschungsgesellschaft (DFG) via RTG
2131 \emph{High-dimensional Phenomena in Probability -- Fluctuations and Discontinuity} and through the research grant 
with the title 
\emph{Higher order numerical approximation methods 
for stochastic partial differential equations}
(Number 175699)
from the Swiss National Science Foundation (SNSF).
One of the author thanks the (Labex CEMPI) Centre Europ\'een pour les Math\'ematiques, la Physique et leurs interactions (ANR-11-LABX-0007-01) for financial support.

\addcontentsline{toc}{section}{Bibliography}
\bibliographystyle{abbrvnat}
\bibliography{bibfile}

\end{document}